\documentclass[a4paper,11pt]{article}
\usepackage[english]{babel}
\usepackage{amsthm}
\usepackage[left=2.7cm,right=2.7cm,top=3cm,bottom=3cm]{geometry}
\usepackage{graphicx}
\usepackage{mathtools}
\usepackage{amsmath}
\usepackage{amssymb}
\usepackage{tikz}
\usepackage{dsfont}
\usepackage{faktor}
\usepackage{mathrsfs}
\usepackage{xcolor}
\usepackage{stmaryrd}
\usepackage[utf8]{inputenc}

\renewcommand{\geq}{\geqslant}
\renewcommand{\leq}{\leqslant}

\def\RR{\mathbb{R}}

\def\NN{\mathbb{N}}
\def\PP{\mathbb{P}}
\def\EE{\mathbb{E}}

\newcommand{\bP}{\mathbf{P}}
\newcommand{\bE}{\mathbf{E}}

\newcommand{\ind}{\mathds{1}}

\newcommand{\gb}{\beta}

\newcommand{\gep}{\varepsilon}       
\newcommand{\gp}{\varphi}

\newcommand{\go}{\omega}

\newcommand{\cE}{\mathcal{E}}

\newcommand{\cR}{\mathcal{R}}
\newcommand{\cC}{\mathcal{C}}

\newcommand{\p}{\mathtt{p}}
\newcommand{\q}{\mathtt{q}}
\newcommand{\dd}{\mathrm{d}}

\newcommand{\inftwo}[2]{\inf_{\substack{#1 \\ #2}}} 
\newcommand{\sumtwo}[2]{\sum_{\substack{#1 \\ #2}}} 

\newcommand{\strength}{\mathrm{Strength}}

\theoremstyle{plain}

\newtheorem{theorem}{Theorem}[section]
\newtheorem{lemma}[theorem]{Lemma}
\newtheorem{proposition}[theorem]{Proposition}
\newtheorem{corollary}[theorem]{Corollary}

\newtheorem{claim}[theorem]{Claim}

\theoremstyle{definition}

\theoremstyle{remark}
\newtheorem{remark}[theorem]{Remark}

\definecolor{Orange}{HTML}{B26536}

\definecolor{Gris}{HTML}{808080}

\numberwithin{equation}{section}

\usepackage{hyperref}
\hypersetup{					
	plainpages=false,			
	colorlinks=true,			
	pdfborder={0 0 0},			
	breaklinks=true,			
	bookmarksnumbered=true,		%
	bookmarksopen=true			%
}

\usepackage{enumitem}
\setlist{itemsep=1pt,topsep=2pt,parsep=1pt, leftmargin=1.5em}

\let\OLDthebibliography\thebibliography
\renewcommand\thebibliography[1]{
  \small
  \OLDthebibliography{#1}
  \setlength{\parskip}{0pt}
  \setlength{\itemsep}{3pt plus 0.3ex}
}

\usepackage{titletoc}
\titlecontents{section}[3em]{}%
    {\thecontentslabel.\enspace}
    {}
    {\titlerule*[0.5pc]{.}\contentspage}

\usepackage{authblk}

\begin{document}

\title{Non-linear conductances of Galton--Watson trees\\ and application to the (near) critical random cluster model}

\author[*]{Irene Ayuso Ventura}
\author[$\dagger$,$\ddagger$]{Quentin Berger}

\affil[*]{\footnotesize Durham University, UK.}
\affil[$\dagger$]{\footnotesize Laboratoire d'Analyse, G\'eom\'etrie et Applications, Universit\'e Sorbonne Paris Nord, France.}
\affil[$\ddagger$]{\footnotesize Institut Universitaire de France.}

\date{}

\maketitle

\begin{abstract}
\noindent 
In this article, we study concave recursions on trees, which appear widely in information theory through algorithms such as belief propagation, and in statistical mechanics through models on tree-like graphs, including the Ising model, percolation, and more generally, the random cluster model.
These tree recursions can, in fact, be compared with non-linear conductances, or $p$-conductances, between the root and the leaves of the tree.
In this article, we estimate the $p$-conductances of $T_n$, a supercritical Galton--Watson tree of depth~$n$, for any $p>1$, for a \textit{quenched} realization of $T_n$.
In particular, we find the sharp asymptotic behavior when~$n$ goes to infinity, which depends on whether the offspring distribution admits a finite moment of order~$q$, where $q=\frac{p}{p-1}$ is the conjugate exponent of~$p$.
We then apply our results to the random cluster model on~$T_n$ (with cluster weight parameter $\q\in (0,2]$ and wired boundary condition) providing sharp estimates on the probability that the root is connected to the leaves.
As an example, for the Ising model on~$T_n$ with plus boundary conditions on the leaves, we find that, at criticality, the quenched magnetization of the root decays like: (i)~$n^{-1/2}$ times an explicit tree-dependent constant if the offspring distribution admits a finite third moment; (ii)~$n^{-1/(\alpha-1)}$ if the offspring distribution has a heavy tail with exponent $\alpha \in (1,3)$.
\end{abstract}

{\small
\setcounter{tocdepth}{1}

\tableofcontents
}

%

\section{Introduction and main results}

We consider a super-critical branching process with reproduction law $\mu$.
We denote by $Z$ a generic random variable with law $\mu$ and we denote ${m:= \bE[Z]}$, that we consider to be finite.
We assume for simplicity that $\mu(0)=0$ so the tree is infinite (and has no leaves) and that $\mu(1)<1$ so the tree is non-degenerate (and $m>1$).

We denote by~$T$ the infinite tree associated with the super-critical branching process, and
for $n\in \NN$, we let $T_n$ be the subtree of depth $n$. 
We equip the tree $T_n$ with a set of resistances $R(e)$ on its edges; if the edge is $uv$ with $u$ the parent of $v$ (we will write $u\rightarrow v$), then we denote $R_v := R(e)$.
Our main objective is to estimate the $p$-conductance (or $p$-capacity) of~$T_n$, equipped with resistances $(R_v)_{v\in T_n}$. 
Such quantities arise naturally in the context of the random cluster model with cluster weight \(\q \in (0, 2]\) (including percolation and the Ising model) on a quenched Galton–Watson tree; we refer to Section~\ref{sec:RCM} below for more details (in particular, we also obtain a tree-recursion in the case $\q>2$, but it is not concave and cannot be related to a non-linear conductance).

Let us mention that the usual effective resistances (or conductances) of random trees have already been considered in~\cite{ABBL09,ChenHuLin}, with a specific choice of (random) resistances $(R_v)_{v\in T_n}$ which corresponds to a \textit{critical} case.
Our work can therefore be seen as a generalization of~\cite{ABBL09,ChenHuLin} to the case of $p$-resistances, with a wider range of (non-random) resistances; we improve some of their  results and present some applications, in particular to the (near) critical percolation or Ising model on a quenched tree.

More generally, we study (concave) recursions on a Galton--Watson tree,  which may arise in the context of statistical mechanics models on tree-like graphs (see~\cite{vdHStFlour} for an overview), or more broadly in belief propagation or population dynamics algorithms (see e.g.\ \cite{MeMo09} for an overview). 
These tree-recursions are often considered in a distributional sense rather than with a fixed tree, corresponding to an annealed setting for random graphs (\textit{i.e.}\ the randomness of the tree is part of the recursion); then, one of the main question is that of the convergence to the fixed point distribution, and we refer for instance to~\cite{FLOC23,MOC19} for recent (and general) results.
Our interest here is slightly different, since we study iterations on a \textit{quenched} (\textit{i.e.}\ fixed) Galton--Watson tree, focusing on the case where the distributional fixed point of the iteration is degenerate, equal to $0$: our main objective is then to estimate precisely the decay rate of the recursion towards zero.
As an application of our results, we derive sharp estimates for the connection (or survival) probability in the sub-critical and (near) critical random cluster model with cluster weight $\q\in (0,2]$ on a quenched Galton--Watson tree.

\subsection{Non-linear ($L^p$) resistive networks}

Let us consider a graph $G=(V,E)$ equipped with a set of (non-negative) resistances $(R(e))_{e\in E}$ on its \textit{non-directed} edges $e= \{x,y\}$ with $x,y\in V$. 
We will also need to consider \textit{directed} edges, that are given by ordered couples $\vec{xy}:=(x,y)$ with $\{x,y\}\in E$; we write $\vec{e} = \vec{xy}$ and $-\vec{e}= \vec{yx}$.
A general theory of non-linear resistances and {conductances} is by now well-developed, and are usually defined through discrete nonlinear potential theory, see for instance~\cite{Soardi06} for an overview.
Here, we focus on a specific non-linear case,  so-called ``$L^p$ resistive networks''.
We will give the definitions directly in terms of the $L^p$-Thomson's principle, since it is  the only tool we need for this article; we refer to \cite{CNS21} for a detailed review of $L^p$-resistances and conductances (or simply $p$-resistances and $p$-conductances), see in particular \cite[Thm.~2.13]{CNS21} for the $L^p$-Thomson's principle.

For $A, Z$ two disjoint subsets of $V$, we consider a \textit{flow} $\theta$ between $A$ (the source) and $Z$ (the sink), which is a \(\RR\)-valued function \(\theta\) on directed edges that verifies $\theta(-\vec{e}) =-\theta(\vec{e})$ and Kirchoff's node law: for any $x \in V \setminus (A\cup Z)$, $\sum_{y: y\sim x} \theta (\vec{xy}) =0$.
The strength of the flow is then defined as
\[
\strength(\theta) = \sumtwo{x \in A, y\notin A}{x\sim y} \theta ( \vec{xy})\,,
\]
and we say that $\theta$ is a flow from $A$ to $Z$ if $\strength(\theta)\geq 0$; we also say that $\theta$ is a \textit{unit} flow if $\strength(\theta)=1$.

For $p>1$, we define the $L^p$-energy (or simply $p$-energy) of a flow $\theta$ from $A$ to~$Z$ as 
\[
\cE_p(\theta) =  \sum_{e\in E}  R(e)^{\frac{1}{p-1}} |\theta(e)|^{\frac{p}{p-1}}  \,,
\]
where $|\theta(e)|$ does not depend on the orientation of the edge since $\theta(-\vec{e}) = -\theta(\vec{e})$.
Then, the $p$-resistance and $p$-conductance between $A$ and~$Z$ are defined, through Thomson's principle, as follows:
\begin{equation}
\label{def:RpCp}
\begin{split}
\cR_p(A \leftrightarrow Z) &:= \inftwo{\theta : A\to Z}{\strength(\theta)=1} \cE_p(\theta)^{p-1}  \\
\text{ and } \quad
\cC_p(A\leftrightarrow Z) & := \cR_p(A \leftrightarrow Z)^{-1} \,. 
\end{split}
\end{equation}
 
It is actually notationally convenient to introduce the conjugate exponent $q=\frac{p}{p-1}$ (\textit{i.e.}\ such that $\frac1p+\frac1q=1$), so that the Thomson's principle~\eqref{def:RpCp} can be rewritten as
\begin{equation}
\label{def:thomson}
\cR_p(A \leftrightarrow Z)^{s} = \inftwo{\theta : A\to Z}{\strength(\theta)=1} \hspace{1mm} \sum_{e\in E}  R(e)^{s} |\theta(e)|^{q}   \,, \qquad \text{ with }\ s:=q-1 = \frac{1}{p-1} \,. 
\end{equation}
Let us also stress that for $p=2$ (that is, $q=2$ and $s=1$), the $p$-resistance and $p$-conductance amount to the usual (linear) effective resistance and conductance.

\begin{remark}
\label{rem:seriesparallel} 
From the definition \eqref{def:RpCp}, we can easily deduce the Series and Parallel laws for  $L^p$ resistive networks: recalling that $s=\frac{1}{p-1}$, we can formally write them as
\[
\label{eq:seriesparallel}
\begin{split}
\cR_p\big(  \mathop{\rule{0.5cm}{1pt}}^{R_1} \mathop{\rule{0.5cm}{1pt}}^{R_2} \big)^{s}& = 
\cR_p\big(  \stackrel{R_1}{\rule{0.5cm}{1pt}} \big)^{ s} + \cR_p\big(  \stackrel{R_2}{\rule{0.5cm}{1pt}} \big)^{ s}  \,, \\
\cC_p\Big(   \begin{array}{c}\displaystyle \mathop{\rule{0.5cm}{1pt}}^{R_1} \\[-10pt]
\displaystyle\mathop{\rule{0.5cm}{1pt}}_{R_2} \end{array} \Big) 
& =  \cC_p\big( \mathop{\rule{0.5cm}{1pt}}^{R_1}\big) + \cC_p\big( \mathop{\rule{0.5cm}{1pt}}^{R_2} \big) \,.
\end{split}
\]
\end{remark}

\begin{remark}
General non-linear networks can be defined by replacing the function ${x\mapsto |x|^{p}}$ by a strictly convex function $\gp: x\mapsto \gp(x)$  (in this case, Thomson's principle becomes a bit harder to state); we refer to \cite{DMS90,Soardi06} and references therein.
Most of our results would also hold for such generalized conductances.
We have chosen to restrict ourselves to the case of \(p\)-conductances for simplicity, and because the applications we have in mind do not require more generality.
\end{remark}

\subsection{Concave recursions on trees}

\paragraph*{The conductance recursion}

It turns out that on trees, one can express the effective conductance between the root and the leaves $\cC_p(\rho \leftrightarrow \partial T_n)$ through a simple recursion, as follows.

For $u \in T_n$, let us consider the subtree $T_n(u)$ of all descendants of~$u$ inside~$T_n$.
We consider $\cC_p(u \leftrightarrow \partial T_n(u))$ the $p$-conductance of the subtree $T_n(u)$ as defined in~\eqref{def:RpCp}, equal by convention to \(+\infty\) if \(u\in \partial T_n\), and we denote 
\[
C_n^{(p)}(u) := R_u \cC_p(u \leftrightarrow \partial T_n(u))\,.
\]
Let us note that \(C_n^{(p)}(u)\) corresponds to the $p$-conductance of the subtree $T_n(u)$ between its root \(u\) and its leaves \(\partial T_n(u)\), when equipped with resistances $(\frac{R_v}{R_u})_{v \in T_n(u)}$. 
Then, the Series and Parallel laws from Remark~\ref{eq:seriesparallel} easily yield the following relation
\begin{equation}
\label{def:conductancerec}
C_n^{(p)}(u) = \sum_{v \leftarrow u} \frac{R_u}{R_v}  \frac{C_n^{(p)}(v)}{(1+C_n^{(p)}(v)^s)^{1/s}}\,, \quad\text{with }  s:=\frac{1}{p-1} >0 \,.
\end{equation}

\begin{remark}
The recursion~\eqref{def:conductancerec} appears in \cite[Lem.~3.1]{PemPer10}, but with $s=p-1$ instead of $s=\frac{1}{p-1}$. 
This is just a convention, which simply means that all results in~\cite{PemPer10} are stated with the conjugate exponent; in practice the $p$-capacity in~\cite{PemPer10}  (denoted by $\mathrm{cap}_p$) should refer to the $\frac{p}{p-1}$-conductance. 
Let us also mention that the term capacity is used in~\cite{PemPer10}, but in our context it turns out to be equivalent to the effective conductance between the root and the leaves thanks to the Dirichlet and Thomson principles (see~\cite[Sec.~2, Thms.~2.10 and 2.14]{CNS21} where this is spelled out in the context of \(L^p\) resistive networks).
%
However, in other contexts, \textit{capacity} might refer to the min-cut value, which typically differs from the conductance (except in trees).
To avoid any ambiguity, we restrict ourselves to the term conductances.
\end{remark}

\paragraph*{General recursions and comparison argument}

Consider the iteration on the rooted tree~$T_n$ given by
\begin{equation}
\label{def:concaverec0}
B_n(u) = \sum_{v \leftarrow u} f_v\big( B_n(v)\big)\,,
\end{equation}
where {$f_v: \RR_+\to \RR_+$ are non-negative functions indexed by $v\in T_n$};
recall also that $v \leftarrow u$ means that $v$ is a descendant of $u$ in $T_n$.
Note that this is very similar to the iteration~\eqref{def:conductancerec}, if one could take $f_v(x) = \frac{R_u}{R_v} g_s(x)$ with the function $g_s(x):=\frac{x}{(1+x^s)^{1/s}}$.

We will soon focus on the case where $f_v \equiv f$ does not depend on $v$, but let us keep the dependence on $v$ for now, to follow the footsteps of~\cite[\S1.3]{PemPer10}.
We will further work with functions \((f_v)_{v\in T_n}\) that are concave and bounded, and that satisfy the same type of expansion\footnote{We use a standard notation: $\Theta$ is a function verifying $0< \liminf \frac{\Theta(\gep)}{\gep} \leq \limsup \frac{\Theta(\gep)}{\gep} <\infty$.} $f_v(x) = K_v x  - \Theta(x^{1+s})$ as $x\downarrow 0$, for some universal \(s>0\) and some $K_v>0$ that may depend on~$v$ --- another formulation is through the comparison bounds~\eqref{eq:encadrement0} below.

Recursions like~\eqref{def:concaverec} appear in several statistical mechanics models on trees, in particular for the random cluster model with cluster weight $\q \in(0,2]$ (where~$\q=1$ corresponds to the percolation model and $\q=2$ to the Ising model), we refer to Section~\ref{sec:RCM} for details.
They can actually be compared with $p$-conductances recursions on trees.
Indeed, we can relate~\eqref{def:concaverec0} to~\eqref{def:conductancerec}, provided that the functions $f_v$ can be suitably compared with $g_s$ and that the tree is equipped with appropriate resistances.
The main comparison tool that we need for $f_v$ is the following, given by Theorem 3.2 in~\cite{PemPer10}.

\begin{proposition}
\label{prop:boundCB}
Let $(f_v)_{v\in T_n}$ be non-negative functions.
Assume that there is some $s>0$ and some constants $\kappa_1$, $\kappa_2$ such that the following holds: for each $v\in T_n$ there is some $K_v>0$ such that for all $x>0$
\begin{equation}
\label{eq:encadrement0}
\frac{K_v x}{(1+ \kappa_1 x^s)^{1/s}} \leq  f_v(x) \leq \frac{K_v x}{(1+ \kappa_2 x^s)^{1/s}} \,.
\end{equation}
Define $(B_n(u))_{u\in T_n}$ iteratively as in~\eqref{def:concaverec0}, with initial condition $B_n(u) =+\infty$ for $u\in \partial T_n$.
Then, if we equip the tree~$T_n$ with resistances $R_v := \prod_{\rho \leq u \leq v} K_u^{-1}$, we have, for all $u\in T_n$
\begin{equation}
\label{boundCB}
\kappa_1^{-1/s} C_n^{(p)}(u)  \leq B_n(u) \leq \kappa_2^{-2/s} C_n^{(p)}(u) \,,  \qquad \text{with } p= \frac{1+s}{s} \,,
\end{equation} 
where we recall that $C_n^{(p)}(u) := R_u \cC_p(u \leftrightarrow \partial T_n(u))$.
\end{proposition}

\begin{proof}[Quick proof]
  The proof simply proceed by iteration.
  For $\kappa>0$, let us define $F_\kappa(u) =  \kappa^{-1/s} C_n^{(p)}(u)$, so that by~\eqref{def:conductancerec} we have that for any $v \in T_n\setminus \partial T_n$,
  \[
  F_\kappa(v) = \sum_{v \leftarrow u}  \frac{R_u}{R_v} \frac{F_{\kappa}(u)}{ (1+ \kappa F_{\kappa}(u)^s)^{1/s}} \,.
  \]
  By definition we have $F_{\kappa_1}(u)= B_n(u) =F_{\kappa_2}(u)$ for all $u\in \partial T_n$.
  Then, we clearly get by iteration that $F_{\kappa_1}(u)\leq B_n(u) \leq F_{\kappa_2}(u)$, using~\eqref{eq:encadrement} and the monotonicity of ${x\mapsto \frac{x}{(1+ \kappa x^s)^{1/s}}}$.
\end{proof}

A similar proof establishes the monotonicity of the conductances in~$p$, we state it in the following simple yet interesting lemma (taken from \cite[Lemma~6.9]{AVB23}).

\begin{lemma}
\label{lem:monotoneC}
Let $1<p \leq p'$.
Then, for any $u\in T_n$ we have $C_n^{(p)}(u) \geq C_n^{(p')}(u)$.
In particular, the $p$-conductance of a tree, $C_n^{(p)}(\rho) = \cC_p(\rho \leftrightarrow \partial T_n)$, is non-increasing in $p$, or equivalently non-decreasing in~$s=\frac{1}{p-1}$.
\end{lemma}

\subsection{Main results: $p$-conductances of Galton--Watson trees}

\label{sec:mainresults}

From this point on, we will work with a fixed $p > 1$ so we omit it from the notation, writing $C_n(u)$ in place of $C_n^{(p)}(u)$.
We wish to estimate the $p$-conductance of the Galton--Watson tree $T_n$, that is, the $p$-conductance between the root and the leaves of~$T_n$:
\[
C_n := C_n(\rho)  =\cC_p(\rho \leftrightarrow \partial T_n)\,.
\]
Note that this is a random variable since it depend on the realization of the tree~$T_n$.
In the following, we focus on the case where the tree is equipped with resistances $(R_v)_{v\in T_n}$ that are of the form:
\begin{equation}
\label{def:Re}
 R_v = (R_n)^{-|v|} \,,
\end{equation}
where $|v|$ is the distance between $v$ and the root, and $R =R_n>0$ is a number which is fixed along $T_n$, but that may depend on $n$. 
The choice~\eqref{def:Re} may seem rather restrictive but this is what appears naturally in the statistical mechanics models that we are interested in; then~$R_n$ is related to the inverse temperature of the model (the fact that~$R_n$ may depend on~$n$ will allow us to study the \textit{critical window} for the random cluster model, as will become clear in Section~\ref{sec:RCM}).
It is also closely related to the choice of conductances considered in~\cite{ABBL09,ChenHuLin}, where the authors estimate the standard $2$-conductance (or $2$-resistance) with resistances $R_v = \xi_v\, m^{|v|}$, for i.i.d.\ random variables $(\xi_v)_{v\in T}$ --- here we focus on the case where $\xi_v\equiv 1$ but we allow more generality in the growth rate of resistances in view of the applications we have in mind, see Section~\ref{sec:RCM}.

More generally, we will consider recursions~\eqref{def:concaverec0} with a function \(f_v \equiv R_n g\) that does not depend on \(v\), for some function \(g: \RR_+ \to \RR_+\) that is concave and bounded, with \(g'(0+)=1\) --- the latter is simply a normalization choice. 
More precisely, we study recursions of the type 
\begin{equation}
\label{def:concaverec}
B_n(u) =  R_n \sum_{v \leftarrow u} g\big( B_n(v)\big)\,,
\end{equation}
with boundary condition $B_n(u) \equiv +\infty$ for $u\in \partial T_n$. 
We also denote $B_n:= B_n(\rho)$.
In particular, in view of the recursion~\eqref{def:conductancerec}, we have that ${B_n(u) = C_n(u)}$ if we take the function $g(x)= x (1+x^s)^{-1/s}$.
Our main assumption is that~$g$ is concave and bounded as mentioned, and verifies a bound similar to~\eqref{eq:encadrement0}: there is some $s>0$ and some constants $\kappa_1$, $\kappa_2$ such that for all $x>0$
\begin{equation}
\label{eq:encadrement}
\frac{x}{(1+ \kappa_1 x^s)^{1/s}} \leq  g(x) \leq \frac{x}{(1+ \kappa_2 x^s)^{1/s}} \,.
\end{equation}
This is for instance verified in the random cluster model with cluster weight~$\q$ (with $s=2$ if ${\q\in (0,2)}$ and $s=3$ for $\q = 2$), see Section~\ref{sec:RCM} below.

We emphasize once again that the resistances are given by $R_v = (R_n)^{-|v|}$, where $R_n$ is a parameter that remains fixed throughout the tree $T_n$, but may depend on $n$.
In particular, except when $R_n\equiv R$, the recursion~\eqref{def:concaverec} that defines~$B_n$ is different from the one that defines~$B_{n+1}$. In fact, one may think of~$B_n$ as belonging to a (tree-)triangular array of recursions $((B_n(u))_{u\in T_n}, n\geq 0)$.

Before we state our results, let us introduce the following normalizing sequence, which plays an important role:
\begin{equation}
\label{def:an}
a_n := \sum_{k=1}^{n} (m R_n)^{-ks} \,.
\end{equation}
To anticipate a bit, let us notice that we have
$a_n = f_n(mR_n)$, with $f_n(x) =  x^{-s} \frac{x^{-ns}-1}{x^{-s}-1}$.
In particular, we easily get that\footnote{We used the standard notation $a_n\asymp a_n'$ if the ratio $a_n/a_n'$ is bounded from above and from below by two universal constants.}
\begin{equation}
\label{eq:orderan}
\begin{array}{lll}
\text{ if } mR_n \in \big(0,1-\tfrac1n \big] & \quad
a_n \asymp  (1- (mR_n)^s)^{-1} (mR_n)^{ -ns} \,,  \\[2pt]
\text{ if } mR_n \in \big[1-\tfrac1n,1+\tfrac1n \big] & \quad
a_n \asymp  n \,,  \\[2pt]
\text{ if } mR_n \in \big[1+\tfrac1n,2] & \quad
a_n \asymp ( (mR_n)^s -1)^{-1} \,.
\end{array}
\end{equation}

Note that the first line also includes the case where $R_n \to 0$, but we will mostly focus on the case where $\liminf_{n\to\infty} R_n >0$.
We refer to the regime $mR_n \equiv 1$ as the \emph{critical case} and one can interpret $|mR_n -1| = O(\frac1n)$ as the \textit{critical window}.

We obtain different results depending on whether $Z\sim \mu$ admits or not a finite moment of order~$q:=s+1$; similarly to what happens for the critical behavior of the Ising model on random trees, see~\cite{DomGiaRvdH14}.
We start with the case where $\bE[Z^q]<+\infty$, where our results are much sharper; we then turn to the case where $\bE[Z^q]=+\infty$.

\subsubsection{Case of a finite moment of order $q=s+1$}

Our first result gives a general estimate on $\bE[B_n]$; in particular, thanks to Proposition~\ref{prop:boundCB}, we may focus on the $p$-conductance $C_n$ of the tree $T_n$, with \(p= \frac{s+1}{s}\).
We first present a result in the case where the offspring distribution has a finite moment of order $q:= \frac{p}{p-1}=s+1$ (note that $q$ is the conjugate exponent of \(p\)), and subsequently derive bounds under the weaker assumption of a finite moment of order $r < q$.

\begin{theorem}
\label{thm:expect}
Let $p>1$ and let $s=\frac{1}{p-1}$.
Assume that $Z\sim \mu$ admits a finite moment of order $q=\frac{p}{p-1} =s+1$.
Then, there is a constant $c_p$ (that depend only on the law $\mu$ and on $p$) such that
\[
c_p (a_n)^{-1/s}\leq \bE[C_n]  \leq (a_n)^{-1/s} \,,
\]
where $a_n$ is defined in~\eqref{def:an}.
Additionally, $( a_n^{1/s} C_n)_{n\geq 1}$ is tight in $(0,+\infty)$.
\end{theorem}

\begin{remark}
Let us note that when $Z$ only admits a moment of order $r<q$, then Lemma~\ref{lem:monotoneC} (note that the monotonicity is reversed when considering the conjugate exponents) combined with Theorem~\ref{thm:expect} gives the bound
\[
\bE[\cC_p(\rho \leftrightarrow \partial T_n)] \geq \bE[\cC_{p'}(\rho \leftrightarrow \partial T_n)] \geq (\tilde a_n)^{1/(r-1)} \,,
\]
with $p'= \frac{r}{r-1}$ the conjugate exponent of $r$ and $\tilde a_n = \sum_{k=1}^n (mR_n)^{-k (r-1)}$. In particular, if $R_n =m^{-1}$, $\bE[C_n] \geq n^{-1/(r-1)}$.
\end{remark}

In view of~\eqref{eq:orderan}, we have the following estimates:
\begin{equation}
\label{eq:orderECn}
\begin{array}{lll}
\text{ if } mR_n \in \big(0,1-\tfrac1n \big] & \quad \bE[B_n] \asymp (1- (mR_n)^s)^{1/s} (mR_n)^{ n} \,, \\[2pt]
\text{ if } mR_n \in \big[1-\tfrac1n,1+\tfrac1n \big] & \quad  \bE[B_n] \asymp n^{-1/s} = n^{-(p-1)} \,, \\[2pt]
\text{ if } mR_n \in \big[1+\tfrac1n,2] & \quad \bE[B_n] \asymp ( (mR_n)^s -1)^{1/s} \,.
\end{array}
\end{equation}
To complete the picture, note that when $mR_n \geq 2$, $\bE[B_n]$ is of order $(mR_n)^{-1}$.

Our next result shows the $L^q$ convergence of the rescaled conductance (or of the rescaled~$B_n$).

\begin{theorem}
\label{thm:convergence}
Let $s>0$ and assume that $Z\sim \mu$ admits a finite moment of order $q=s+1$.
Let \((B_n(u))_{u\in T_n}\) be defined as in~\eqref{def:concaverec} and assume that~\eqref{eq:encadrement} holds.
Then, if $\limsup_{n\to\infty} mR_n \leq 1$ and $\inf_{n\geq 0} R_n >0$, we have, as $n\to\infty$ 
\[
\frac{B_n}{\bE[B_n]} \xrightarrow[n\to\infty]{L^q} W \,,
\]
where $W := \lim_{n\to\infty} \frac{1}{m^n} Z_n$ is the a.s.\ and $L^q$ limit of the usual martingale associated to the branching process\footnote{If $Z\sim \mu$ admits a finite moment of order~$q$, the $L^q$ convergence of the martingale is non-trivial: we refer to \cite[Prop.~1.3]{Liu01} for a proof.}.
If $R_n \equiv R$ with $R\in (0,m^{-1}]$, then the convergence also holds almost surely.
\end{theorem}

Finally, we estimate precisely the expectation $\bE[B_n]$, under some condition on the function~$g$. We have the following result.

\begin{proposition}
\label{prop:asymp}
Let $s>0$ and assume that $Z\sim \mu$ admits a finite moment of order $q=s+1$.
Let \((B_n(u))_{u\in T_n}\) be defined as in~\eqref{def:concaverec} and assume that~\eqref{eq:encadrement} holds and that the following limit exists
\[
\kappa_g:=\lim_{x\downarrow 0} \frac{1}{x^{s+1}} (x-g(x)) \in (0,+\infty) \,.
\] 
(Note that for $g(x) =\frac{x}{(1+x^s)^{1/s}}$, we have $\kappa_g = s^{-1}$).
Then, if $\lim_{n\to\infty} mR_n = \vartheta \in (0,1]$, we have
\[
\bE[B_n] \sim \alpha_s(\vartheta) \; (a_n)^{-1/s} \,,  \qquad \text{ as } n\to\infty\,,
\] 
for some constant $\alpha_{s}(\vartheta)$, with $\alpha_s(1) := (s \kappa_g \bE[W^{s+1}])^{-1/s}$ in the case $\lim_{n\to\infty} mR_n =1$.
\end{proposition}

Let us mention that Proposition~\ref{prop:asymp} improves in particular the result \cite[Thm.~1.2]{ChenHuLin}, which considers some additional source of randomness on the resistances but treats only the case of linear ($p=q=2$) conductances in the critical case $ mR_n\equiv 1$ and requires a moment of order~$3$, \textit{i.e.}\ $q+1$, for the branching process.

\subsubsection{Case of an infinite moment of order $q=s+1$}

In the case of an infinite moment of order~$q$, we need to assume some (either upper or lower) bounds on the truncated $q$-moment of $Z\sim \mu$.
More precisely, we assume that there is some  $\alpha\in (1,q]$ and some slowly varying function $L(\cdot)$ such that, for $x\geq 1$
\begin{align}
\label{hyp:lower}
\bE[ (Z\wedge x)^q ] \geq  c_1 L(x) x^{q-\alpha} \,,\\
\bE[ (Z\wedge x)^q ] \leq c_2 L(x) x^{q-\alpha} \,,
\label{hyp:upper}
\end{align}
for some constants $c_1, c_2$.
We obtain a lower bound on $\bE[B_n]$ assuming~\eqref{hyp:upper} and an upper bound on $\bE[B_n]$ assuming~\eqref{hyp:lower}.

Let $h$ be a increasing function such that $h(x)\sim L(1/x) x^{\alpha -1}$ as $x\downarrow 0$, and denote  its inverse by $h^{-1}$. 
Then, $h^{-1}(x) \sim \tilde L(x) x^{1/(\alpha-1)}$ as $x\downarrow 0$, for some slowly varying function~$\tilde L$, see~\cite[\S1.5.7]{BGT89}.
In analogy with~\eqref{eq:orderECn},  which corresponds to the case $h(x)= x^{q-1}$, define a normalizing sequence \(\gamma_n\) as follows:
\begin{equation}
\label{def:gamma}
\begin{array}{lll}
\text{ if } mR_n \in \big(0,1-\tfrac1n \big] & \quad \gamma_n =  h^{-1}(1-mR_n) \, (mR_n)^{n} \,, \\[2pt]
\text{ if } mR_n \in \big[1-\tfrac1n,1+\tfrac1n \big] & \quad  \gamma_n = h^{-1}(1/n) \,, \\[2pt]
\text{ if } mR_n \in \big[1+\tfrac1n,+\infty) & \quad \gamma_n = h^{-1}(mR_n-1) \,.
\end{array}
\end{equation}
Let us also define $\tilde{\gamma}_n$ by setting 
\begin{equation}
\label{def:gammatilde}
\begin{array}{ll}
\text{ if } mR_n \in \big(0,1-\tfrac1n \big] & \quad \tilde{\gamma}_n  = h^{-1}(1/n) (n(1-mR_n))^{1/s} (mR_n)^n \,, \\[2pt]
\text{ if } mR_n \in \big(1-\tfrac1n,+\infty) & \quad \tilde{\gamma}_n  =\gamma_n \,.
\end{array}
\end{equation}
In practice, in the case $R_n \equiv R$ and for $L(x) \equiv 1$ in~\eqref{hyp:lower}-\eqref{hyp:upper}, we have:

  (i) $\gamma_n\asymp (mR)^{n}$ and $\tilde{\gamma}_n \asymp n^{\frac{1}{q-1} - \frac{1}{\alpha-1}} \gamma_n$ if $mR<1$;
  
  (ii) $\gamma_n=\tilde{\gamma}_n \asymp n^{-1/(\alpha-1)}$ if $mR=1$;

  (iii) $\gamma_n=\tilde{\gamma}_n \asymp (mR-1)^{1/(\alpha-1)}$ if $mR>1$.


\begin{theorem}
\label{thm:expect2}
Assume that $0<\inf_{n} mR_n \leq \sup_n mR_n \leq K$ for some $K <\infty$, and let $\gamma_n,\tilde \gamma_n$ be as in~\eqref{def:gamma}-\eqref{def:gammatilde}. Then there are constants $c,c'$ such that:
\begin{itemize}
\item If~\eqref{hyp:lower} holds, then $\bE[B_n] \leq c \gamma_n$, and in particular 
\[
\lim_{K\uparrow +\infty} \limsup_{n\to\infty}\bP(B_n \geq K \gamma_n) =0 \,;
\]

\item If~\eqref{hyp:upper} holds, then  $\bE[B_n] \geq c' \tilde{\gamma}_n$, or more precisely
\[
\lim_{\gep\downarrow 0} \limsup_{n\to\infty}\bP(B_n \leq \gep \tilde{\gamma}_n) =0 \,.
\]
\end{itemize}
\end{theorem}

Notice that there is a tiny gap between $\gamma_n$ and $\tilde{\gamma}_n$ in the case where $mR_n\leq 1- \frac1n$; this should be an artifact of the proof, and we believe that the correct decay is given by~$\gamma_n$. 
However, $\gamma_n$ and $\tilde{\gamma}_n$ only differ by a polynomial in $x_n:=n(1-mR_n)$: indeed, by Potter's bound \cite[Thm. 1.5.6]{BGT89} we have that $h^{-1}(1/n) \geq  c_{\delta} (x_n)^{-\delta - \frac{1}{\alpha-1}} h^{-1}(1-mR_n)$.
We therefore obtain that $\gamma_n \geq \tilde{\gamma}_n \geq c_{\delta} (x_n)^{a -\delta} \gamma_n$, with $a=\frac{1}{s} - \frac{1}{\alpha-1}$.
Let us conclude by noting that in the case $mR_n \leqslant1 - 1/n$, both $\gamma_n,\tilde{\gamma}_n$ have the same exponential decay in $(mR_n)^n$.

\begin{remark}
\label{rem:tail-moment}
We mention that we consider an assumption on the truncated $q$-th moment since they are slightly more general than assumptions on the tail of $Z$ as considered for instance in~\cite{DomGiaRvdH14} (see Section~\ref{ssec:RCM-app} for some further comments); also, the truncated moments are the quantities naturally appearing in the proof. 
Consider for instance the following conditions:
\begin{align}
\label{hyp:lowertail}
\bP(Z>x) \geq  c_1' \hat L(x) x^{-\alpha} \,, \\
\bP(Z>x) \leq c_2' \hat L(x) x^{-\alpha} \,,
\label{hyp:uppertail}
\end{align}
for some slowly varying $\hat L$.
Then, one can actually deduce~\eqref{hyp:lower} from~\eqref{hyp:lowertail} (resp.~\eqref{hyp:upper} from~\eqref{hyp:uppertail}), since we have 
$
\bE[ (Z\wedge x)^q ] = q \int_0^x  t^{q-1}\bP(Z>t) \dd t .
$
We then have $L(x) = \hat L(x)$ in the case $ \alpha<q$, whereas $L(x)= \int_1^x u^{-1} \hat L(u) \dd u  \gg \hat L(x)$ (see~\cite[Prop.~1.5.9.a]{BGT89}) in the case $\alpha=q$.

Let us also note that the conditions \eqref{hyp:lower}-\eqref{hyp:upper} are \textit{weaker} than the tail conditions \eqref{hyp:lowertail}-\eqref{hyp:uppertail}.
Indeed, on one hand, we have $\bE[ (Z\wedge x)^q ] \geq x^q \bP(Z>x)$, so~\eqref{hyp:upper} implies the upper bound $\bP(Z>x) \leq c_2 L(x)x^{-\alpha}$, but this is not optimal in the case $\alpha=q$. 
However, in the case $\alpha<q$, \eqref{hyp:upper} is equivalent to~\eqref{hyp:uppertail} with $\hat L=L$.
On the other hand, to obtain a lower bound on~$\bP(Z>x)$, one needs to have $\alpha<q$ and both \eqref{hyp:lower}-\eqref{hyp:upper}; simply write $\bE[(Z\wedge Ax)^q] \leq (Ax)^q\bP(Z>x) + \bE[(Z\wedge x)^q]$.
\end{remark}

\subsection{Organisation of the paper}

In section \ref{sec:RCM} we introduce the main motivation for our results on non-linear conductances: the random cluster model (RCM) on Galton--Watson trees. 
In Section~\ref{ssec:RCM-recursion}, we present the tree recursion arising in the model (we postpone the proof to Appendix~\ref{sec:RCM-proofs}) and, in Section \ref{ssec:RCM-app}, we apply our main results on $p$-conductances to get as a corollary the critical behavior of the RCM on a quenched Galton--Watson tree.

The rest of the paper is divided as follows:
\begin{itemize}
\item In Section~\ref{sec:preliminaries}, we introduce a few preliminary (technical) results on branching processes.
\item In Section~\ref{sec:est-finite} we focus on the estimates of the moments of $C_n$, in the case where $Z \sim \mu$ admits a finite moment of order $q$, and in particular we prove Theorem~\ref{thm:expect}.
\item  In Section~\ref{sec:cvg}, we show the $L^p$ convergence of the normalized conductance of Theorem \ref{thm:convergence} and we conclude the precise estimate of $\bE[B_n]$ of Proposition \ref{prop:asymp}.
\item In Section~\ref{sec:est-infinite}, we prove Theorem~\ref{thm:expect2}, which treats the moments of $C_n$ in the case $Z \sim \mu$ has an infinite moment of order $p$. 
\item In Appendix~\ref{sec:RCM-proofs} we prove the RCM tree recursion presented in Section~\ref{ssec:RCM-recursion}.
\item In Appendix~\ref{app:tail} we provide some technical proofs regarding branching processes with heavy tails, which are necessary for the estimates of Theorem \ref{thm:expect2}.
\end{itemize}

\section{The (critical) random cluster model on a quenched Galton Watson tree}
\label{sec:RCM}

The \textit{Random Cluster Model} (RCM) or \textit{FK-percolation}, introduced by Fortuin and Kasteleyn in~\cite{FK72}, unifies a number of models in statistical physics, including the percolation model, the Ising or $\q$-states Potts models and the uniform spanning tree.
Some of these models, when considered on trees, exhibit recursions as in~\eqref{def:concaverec0}: this is true in particular for the connection (sometimes called \textit{survival}) probability of the RCM on a tree, see~\eqref{eq:defConnectionProba}.
We give a quick presentation of the model below, but we refer to~\cite{Grimmett06} for a more complete overview of the RCM and its connection to other models; see also \cite[Ch.~1]{DC17-PIMS}.
Our goal here is to give the exact asymptotic behavior of the connection probability, both \textit{at} and \textit{near} criticality. We obtain in particular that the so-called critical one-arm exponent is either equal to \(1\) when \(\q\in (0,2)\) (a ``percolation regime'') or to \(1/2\) when \(\q=2\) (for the Ising model, which to our knowledge was studied at criticality only in~\cite{HK18}, on \(d\)-regular trees); note that these exponents agree with the mean-field exponents.

Let us mention that one of the motivation for considering statistical mechanics models on trees is to apply the results to general random tree-like graphs, see e.g.\ \cite{vdHStFlour} for an overview (see also  \cite{DemboMontanari10,DomGiaRvdH} in the context of the Ising model, and \cite{CvdH25,DMSS14,HJP23} for the random cluster model).

\subsection{The random cluster model on a graph}

For a finite graph $G=(V,E)$, the random cluster model is a Gibbs measure on percolation configurations $\go\in \{0,1\}^E$, where in the configuration $\go=(\go_e)_{e\in E}$, an edge $e$ is called \emph{open} if $\go_e=1$ and \textit{closed} if $\go_e=0$.
We also denote $x \stackrel{\go}{\longleftrightarrow} y$ if $x$ and $y$ can be connected by an open path in $\go$.
For two parameters $\p\in [0,1]$, $\q>0$, the RCM measure $\PP_{\p,\q}^G$ is defined by
\begin{equation}
\PP_{\p,\q}^G(\go) = \frac{1}{Z_{\p,\q}^G}   \p^{o(\go)} (1-\p)^{f(\go)} \q^{k(\go)} \,, 
\end{equation}
where $o(\go) = \sum_{e\in E} \go_e$ is the number of open edges, $f(\go) =\sum_{e\in E} 1-\go_e$ is the number of closed edges and $k(\go)$ is the number of \textit{clusters} in $\go$, \textit{i.e.}\ the number of connected components of the subgraph of $G$ induced by $\go$.
Here, $Z_{\p,\q}^G$ is the partition function of the model, that is the constant that normalizes $\PP_{\p,\q}$ to a probability.
Let us mention that when $\q\geq 1$, the model enjoys some monotonicity and correlation inequalities, see e.g\ \cite[Ch.~3]{Grimmett06}; these are essential tools for many of the techniques developed for the study of the RCM, which is therefore often considered only for $\q\geq 1$.

\paragraph*{Relation with percolation and Ising/Potts models}
Notice that for $\q=1$, one clearly recovers the percolation model with parameter $\p$.
On the other hand, for $\q\in \{2,3,4,\ldots\}$ the RCM can be coupled to the (ferromagnetic) $\q$-states Potts model with inverse temperature~$\gb$ verifying $e^{-\gb} =1-\p$, see \cite[\S1.4]{Grimmett06}, defined by the Gibbs measure on configurations $\sigma\in \{1, 2,\ldots, \q\}^{V}$ by
\[
\mu_{\gb,\q} (\sigma) = \frac{1}{\mathcal{Z}_{\gb,\q}} \exp\Big( \gb \sum_{ (x,y)\in E} \ind_{\{\sigma_x=\sigma_y\}} \Big) \,.
\] 
Notice that the case $\q=2$ corresponds to the Ising model up to a change of parameter $\beta$, since then $\ind_{\{\sigma_x=\sigma_y\}} = \frac12 (\sigma_x\sigma_y +1)$.

Additionally, let us stress that in the coupling between the RCM and the $\q$-state Potts model, one has a direct relation between the two-point correlation functions for the Potts model and connection probabilities in the RCM, see \cite[Thm.~1.16]{Grimmett06}:
\begin{equation}
\label{eq:correlconnect}
\tau_{\gb,\q}(x,y):= \mu_{\gb,\q}(\sigma_x=\sigma_y) -\frac{1}{\q} = (1-\q^{-1}) \PP_{\p,\q}\big(x \stackrel{\go}{\longleftrightarrow} y \big) \,.
\end{equation}
Notice here that in the definition of $\tau_{\gb,\q}(x,y)$, one compares the probability that two spins $\sigma_x,\sigma_y$ are equal under the Potts measure $\mu_{\gb,\q}$ with the same probability if $\sigma_x,\sigma_y$ were independent and uniform in $\{1,\ldots, \q\}$.

\paragraph*{Boundary conditions}
Analogously to the Ising and Potts models, one may introduce a boundary condition when considering the random cluster model.
Here, we will only focus on the \emph{wired} boundary condition on the graph, which consists in contracting all vertices of the boundary $\partial G$ into one single vertex denoted $\{\partial G\}$: we denote by $\bar G$ the resulting graph and $\bar \PP_{\p,\q}^{\bar G}$  the resulting RCM.
In terms of the Potts model, this corresponds to putting all spins equal on the boundary; for the Ising model this is the model with plus (or minus) boundary condition.

\subsection{The random cluster model on trees and its associated recursion}
\label{ssec:RCM-recursion}

In the case of a tree~$T_n$ of depth $n$, the boundary $\partial T_n$ corresponds to the leaves and $\bar T_n$ is the tree with all leaves identified to one vertex. 
In view of~\eqref{eq:correlconnect}, the connection (or \textit{survival}) probability
\[
\label{eq:defConnectionProba}
\pi_n = \pi_n^{\p,\q}(\rho) := \bar \PP_{\p,\q}^{\bar T_n}(\rho \leftrightarrow \partial T_n)
\] 
is one of the key quantities of interest; note that it depends on the realization of the tree~$T$ and we will consider it under a quenched realization of~$T$.
In particular, in the case where $\q =2$, one has that~$\pi_n$ is \emph{equal} to the magnetization of the root in the Ising model on $T_n$ with plus boundary condition on the leaves $\partial T_n$; a similar interpretation holds for $\pi_n$ in the $\q$-Potts model if $\q \in \{2,3,\ldots\}$, see~\eqref{eq:correlconnect}.
In particular, one is interested in knowing whether, depending on the parameters $\p,\q$, the connection probability $\pi_n$ vanishes or not as $n\to\infty$, and if so, at what rate.

Let us mention that the RCM on $d$-ary (or Cayley) trees has been studied in several articles, see e.g.~\cite[\S10.9-10]{Grimmett06} and references therein.
In a nutshell, the results from~\cite{Hagg96} (\cite{Grimmett06} only treats the case $d=2$) state that for every fixed $\q>0$ a phase transition occurs at some $\p_c(q)\in (0,1)$ given by different formulas according to whether $\q\leq 2$ or $\q>2$: if $\q\in(0,2]$, then $\p_c(\q)  = \frac{\q}{d+\q-1}$ and the phase transition is continuous; if $\q>2$, then $\p_q(\q)$ is the unique value of $\p$ such that the polynomial $Q_{\p,\q}(x):=(\q-1)x^{d+1} - (\q-1+ \frac{\p}{1-\p})x^d  + \frac{1}{1-\p}x -1$ has a double root in $(0,1)$ (note that for $d=2$, this gives $\p_c(\q) = \frac{2\sqrt{\q-1}}{1+2\sqrt{\q-1}}$), and the phase transition is discontinuous.

The random cluster model seems to have been considered on \emph{$d$-regular} (random) graphs, see e.g.\ \cite{BDS23,BBC23,HJP23}, which are locally $d$-ary trees.
However, it does not appear to have been investigated on more general random graphs without fixed degree. 
In particular, to the best of our knowledge, it has not been studied on random trees, such as Galton–Watson trees.
Our goal is to show that the connection probabilities satisfy a recursive relation on trees and to apply our results of Section~\ref{sec:mainresults} to obtain information on the critical (or near and sub-critical) random cluster model.

\subsubsection{The RCM iteration on trees}

Let $T_n$ be a tree of depth $n$ with leaves only at generation~$n$.
We then define, for $u\in T_n$
\[
\pi_n(u) = \pi_n^{\p,\q}(u) := \bar \PP_{\p,\q}^{\bar T_n(u)} \big( u \leftrightarrow \partial T_n(u)\big) \,
\]
the connection probability from the root to the leaves for the RCM inside the subtree $T_n(u)$, with \emph{wired} boundary conditions.
Then, we have the following recursion on the tree $T_n$, whose proof we present in Appendix~\ref{sec:RCM-proofs}.

\begin{proposition}
\label{prop:RCM}
If $u\in \partial T_n$, we have $\pi_n(u) =1$.
If $u\notin\partial T_n$, then we have the following recursion: 
\begin{equation}
\label{rec:pinu}
\gp_\q(\pi_n(u)) = \prod_{y \leftarrow u} \gp_\q\big(  \gamma\, \pi_n(v)\big) \quad \text{ where  }\  \gp_{\q}(x):=\frac{1-x}{1+(\q-1)x}  
\end{equation}
and $\gamma= \gamma_{\p,\q} := \gp_{\q}(1-\p)$.
Setting $B_n(u) := -\log \gp_\q(\pi_n(u))$, we have $B_n(u)=+\infty$ if $u\in \partial T_n$ and, if $u\notin \partial T_n$, the following recursion holds: letting \( \gb \) be such that \( 1 - \p = e^{-\gb} \),
\begin{equation}
\label{def:RCMrec}
B_n(u) =  \sum_{v\leftarrow u} \psi_{\q}^{-1}  \big( \psi_\q(\gb) \psi_\q(B_n(v)) \big) \,,\quad \text{ where } \
\psi_{\q}(x) := \frac{e^{x}-1}{e^{x} +\q-1} = \gp_\q(e^{-x}) \,. 
\end{equation}
\end{proposition}

Let us stress that when $\q=2$, we have that $\psi_{\q}(x) = \tanh(x)$, so one recovers the well-known Lyons' iteration for the Ising model, see~\cite{Lyons89} (see also \cite[Lem.~2.3]{DemboMontanari10} for the specific form $ \psi_{\q}^{-1}  \big( \psi_\q(\gb) \psi_\q (x) \big) = \tanh^{-1}( \tanh(\gb) \tanh(x))$ we have here).

\begin{remark}
As a side remark, let us note that the function $\gp_\q: x \in [0,1]  \mapsto \frac{1-x}{1+(\q-1)x}$ is an involution, \textit{i.e.}\ $\gp_\q\circ\gp_\q (x) =x$ for any $x\in [0,1]$.
\end{remark}

\subsubsection{Properties of the recursion}

Let us fix $\q >0$ and, for every $\gb>0$, let us introduce the function 
\[
\label{def:gRCM}
g_{\gb}(x) = g_{\gb,\q}(x) := \frac{1}{\psi_\q(\gb)} \psi_\q^{-1}  \big( \psi_\q(\gb) \psi_\q(x) \big)\,.
\]
The normalization by $\psi_\q(\gb)$ is here to ensure that $g_{\gb}'(0)=1$.
With this notation, the recursion~\eqref{def:RCMrec} becomes
\begin{equation}
\label{def:RCMrec2}
B_n(u) = \psi_\q(\gb)  \sum_{v\leftarrow u} g_{\gb} (B_n(v)) \,,
\end{equation}
and is similar to~\eqref{def:concaverec0}  with resistances $R_v := \psi_{\q}(\gb)^{-|v|}$, \textit{i.e.}~\eqref{def:concaverec} with $R_n=\psi_\q(\gb)$.

Let us now stress that the function $g_{\gb}$ is concave on $\RR_+$ if and only if $\q\in (0,2]$; we will therefore restrict our attention to the case $\q\in (0,2]$.
As a clue for this fact, notice that $g_{\gb}''(0) = \frac{\q-2}{\q} (1- \psi_\q(\gb))$, which is negative if $\q<2$, equal to zero if $\q=2$ and positive if~$\q>2$.

Notice also that $g_{\gb}$ is bounded and that, as $x\downarrow 0$
\begin{equation}
\label{eq:expansiong}
\begin{split}
\text{ if } \q\in (0,2)\,, &\qquad g_{\gb}(x) = x -  \frac{2 - \q}{2\q} (1- \psi_\q(\gb)) x^2 (1+o(1)) \\
\text{ if } \q =2\,,  &\qquad g_{\gb}(x) = x - \frac{1}{12} (1- \psi_\q(\gb)^2) x^3 (1+o(1))
\end{split}
\end{equation}
so we indeed have~\eqref{eq:encadrement} with $s=1$ if $\q\in (0,2)$ and with $s=2$ if $\q=2$; this was noticed in~\cite{PemPer10} in the case of the Ising model (with plus boundary condition, which correspond to our wired boundary condition for the RCM).

\begin{remark}
Let us stress that a similar recursion also arises for the return probability of a random walk on a tree $T_n$ equipped with resistances $(R_v)_{v\in T_n}$.
Consider the ``non-descending'' probability
\[
\pi_n(u) = \bP_u\big( \tau_{\underline{u}} < \tau_{\partial T_n(u)}\big) \,,
\]
where $\underline{u}$ is the parent of $u$ and $\tau_A$ is the hitting time of the set $A$.
The quantities $\pi_n(u)$ help encode the recurrence/transience of the graph and can be useful to estimate the Green function on the tree.
Then, setting $b_n(u) = \frac{1-\pi_n(u)}{\pi_n(u)}$ (in particular $\pi_n(u)=(1+b_n)^{-1}$), one may verify that one gets the recursion
\[
b_n(u) = \sum_{v\leftarrow u} \frac{R_v}{R_u} \frac{b_n(v)}{1+b_n(v)} \,,
\]
so $b_n(v)$ is exactly the $2$-conductance, $b_n(u) = \cC_2( u \leftrightarrow \partial T_n(u))$.
\end{remark}

\subsection{Main results: critical point and critical behavior}
\label{ssec:RCM-app}

\subsubsection{The critical point and critical behavior at (and below) criticality}
 
For any \textit{fixed} $\gb$ (recall $e^{-\gb} =1-\p$), we have that~\eqref{def:RCMrec2} is exactly the recursion~\eqref{def:concaverec} with $R_n \equiv R = \psi_\q(\gb)$, and~\eqref{eq:encadrement} holds with $s=1$ if $\q\in (0,2)$ and with $s=2$ if $\q=2$.
First of all, Theorems~\ref{thm:expect}-\ref{thm:expect2} allows us to identify the critical point.

For any $\gb>\gb_c$ (\textit{i.e.} $\p=1-e^{-\gb} >\p_c$), define 
\begin{equation}
\label{def:pi}
\pi_n(\gb)= \pi_n^{\q}(\gb,T) := \bar \PP_{\p,\q}^{\bar T_n}(\rho \leftrightarrow \partial T_n)\,,
\qquad 
\pi(\gb)  = \lim_{n\to\infty} \pi_n(\gb) \,,
\end{equation}
where the limit exists by monotonicity.

Recall that $B_n= -\log \gp_\q(\pi_n)$, so in particular $\pi_n$ goes to $0$ if and only if $B_n$ goes to $0$.

\begin{theorem}
\label{thm:RCMtransition}
Let $\q\in (0,2]$, define $\pi_n= \bar \PP_{\p,\q}^{\bar T_n}(\rho \leftrightarrow \partial T_n)$ the connection (or survival) probability, and let $e^{-\gb} =1-\p$.
Let $\gb_c$ be defined by the relation $m\psi_\q (\gb_c)=1$, with $\psi_\q$ in~\eqref{def:RCMrec}.
Then we have that, in $\bP$-probability,
\[
\lim_{n\to\infty} \pi_n(\gb) =0  \quad \text{ if }\gb \leq \gb_c
\qquad \text{ and } \qquad 
\liminf_{n\to\infty} \pi_n(\gb) >0  \quad \text{ if }\gb > \gb_c\,.
\]
In particular, $\pi(\gb) >0$ if and only if $\gb>\gb_c$.
\end{theorem}

\begin{remark}
\label{rem:pointcritique}
Using the definition of $\psi_\q(\gb)$ and using that $\gp_\q$ is an involution, we can rewrite the relation $m\psi_\q(\gb_c)=1$ as $e^{-\gb_c} = \gp_\q(\frac1m) = \frac{m-1}{m+\q-1}$, or equivalently, ${\p_c = 1-e^{-\gb_c} = \frac{\q}{\q+m-1}}$.
\end{remark}

Theorem~\ref{thm:RCMtransition} therefore identifies the critical point $\p_c =1-e^{-\gb_c}$ for the random cluster model with $\q\in (0,2]$ on a (random) Galton--Watson tree.
The case $\q>2$ is understood in the case of a $d$-ary tree but remains somehow mysterious on a random tree and let us briefly comment on it.
In fact, when $\q>2$ the phase transition on a $d$-ary tree happens at a certain $\p_c(\q)$ characterized by being the only (double) root in $(0,1)$ of a certain polynomial $Q_{\p,\q}^{(d)}(x)$ (see \ref{ssec:RCM-recursion}).
In terms of the recursion \eqref{def:RCMrec2}, this correspond to finding $\p$ such that $d g_{\beta,\q}(x) =x$ has only one solution (see Section~\ref{def:gRCM}).
For the RCM on a Galton--Watson tree, if $\q>2$, the \emph{random} recursion \eqref{def:RCMrec2} is still valid but the function $g_{\beta}$ is then not concave anymore, and it is not clear whether a characterization of $p_{c}(\q)$ in terms of a certain (averaged?) polynomial remains valid.

Additionally, our results of Section~\ref{sec:mainresults} directly give precise estimates for the critical case $\gb=\gb_c$. Let us stress that the survival probability on a quenched Galton--Watson tree has recently been studied for percolation (\textit{i.e.}\ $\q=1$), see \cite{AV23,Mich19}, where the results are more precise in the case of a heavy-tail but the techniques used are very different.
On the other hand, the result for the Ising model (or for the random cluster model with $\q\neq 1$) does not seem to be known.

The following result is a direct consequence of Theorem~\ref{thm:convergence}-Proposition~\ref{prop:asymp} in the case where $Z\sim \mu$ admits a finite moment of order $s+1$ and of Theorem~\ref{thm:expect2} in the case where $\bE[Z^{s+1}]=+\infty$.
Notice that $-\log \gp_\q(x) \sim \q x$ as $x\downarrow 0$, so $\pi_n \sim \q^{-1} B_n$ when $B_n$ goes to $0$.
Additionally, note that from~\eqref{eq:expansiong} we get that $\kappa_g:=\lim_{x\downarrow 0} \frac{1}{x^2} (x-g_{\gb}(x))= \frac{2-\q}{2\q} (1- \psi_\q(\gb))$ if $\q\in(0,2)$ and $\kappa_g:=\lim_{x\downarrow 0} \frac{1}{x^3}(x-g_{\gb}(x)) = \frac{1}{12} (1- \psi_\q(\gb)^2)$ if $\q=2$.

\begin{theorem}[Critical case]
\label{thm:RCMcritical}
Let $\gb=\gb_c$, with $\gb_c$ defined by the relation $m\psi_\q (\gb_c)=1$.
We then have the following asymptotic results for $\pi_n$.
Letting $s=1$ in the case $\q\in(0,2)$ and $s=2$ in the case $\q=2$, we have the following:
\begin{itemize}
\item if $\bE[Z^{1+s}]<+\infty$, we have, almost surely and in $L^{1+s}$,
\[
\lim_{n\to\infty} n^{1/s} \pi_n= \alpha_{\q} W \qquad \text{ with } \ \alpha_{\q} := 
\begin{cases}  \frac{2}{2-\q} \frac{m}{m-1} \bE[W^2]^{-1} & \text{ if } \q\in(0,2) \,, \\[3pt]
\frac{\sqrt{3}}{\sqrt{2}}  \frac{m}{\sqrt{m^2-1}} \bE[W^3]^{-1/2} &  \text{ if } \q=2
\,.
\end{cases}
\]
\item if $c_1 L(x) x^{1+s-\alpha}\leq \bE[(Z\wedge x)^{1+s}] \leq c_2 L(x) x^{1+s-\alpha}$ for some $\alpha\in (1,1+s]$, some slowly varying function $L(\cdot)$ and some constants $c_1,c_2>0$, then 
\[
\Big( \frac{\pi_n}{h^{-1}(1/n)} \Big)_{n\geq 1} \quad \text{  is tight in $(0,+\infty)$},
\] 
with $h^{-1}(\cdot)$ an asymptotic inverse of $L(1/x) x^{\alpha-1}$ as $x\downarrow 0$.
\end{itemize}
\end{theorem}

Let us also mention that by classical calculations one has \(\bE[W^2] = \frac{1}{m(m-1)} \bE[Z(Z-1)]\). Hence, when \(\q\in (0,2)\) we can rewrite the constant \(\alpha_\q\) as \(\alpha_{\q} = \frac{2}{2-\q} \frac{m^2}{\bE[Z(Z-1)]}\), matching that of~\cite{AV23} in the case \(\q=1\) of percolation.

We complete the above result with the subcritical case $\gb<\gb_c$: our results imply a sharp estimate on the exponential decay rate of the survival probability.
We can summarize our results as follows.

\begin{theorem}[Sub-critical case]
Let  $\q\in (0,2]$ and let $\gb_c$ defined by the relation $m\psi_\q (\gb_c)=1$.
Then, if $\gb<\gb_c$, we have that $\lim_{n\to\infty} (\pi_n)^{1/n} = m \psi_\q(\gb)<1$ in $\bP$-probability.
\end{theorem}

In fact, Section~\ref{sec:mainresults} gives much more precise results (see Theorem~\ref{thm:convergence}-Proposition~\ref{prop:asymp} and Theorem~\ref{thm:expect2}), that we describe informally as follows: letting $s=1$ if $\q\in(0,2)$ and $s=2$ if $\q =2$, we have
\begin{itemize}
\item if $\bE[Z^{1+s}]<+\infty$, we have $ \lim_{n\to\infty}  (m\psi_\q(\beta))^{-n} \pi_n  = \alpha_\q(\gb) W$ for some constant $\alpha_\q(\gb)$;

\item if $\bE[(Z\wedge x)^{1+s}] = n^{c+o(1)}$ for some constant $c$, then $ C \geq  (m\psi_\q(\beta))^{-n} \pi_n \geq n^{-c'}$ as~${n\to\infty}$.
\end{itemize}

\subsubsection{Near-supercritical connection probability}

Once the critical point $\gb_c$ is identified, it is also natural to consider the \emph{near-critical regime}, \textit{i.e.}\ take an inverse temperature that may depend on $n$, $\gb = \gb_n$, with $ \lim_{n\to\infty}\gb_n = \gb_c$.
This is why we allowed the resistances $R_n = \psi_{\q}(\gb_n)$ in~\eqref{def:concaverec} to also depend on $n$. 
In fact, Theorem~\ref{thm:RCMcritical} can be generalized to the whole near-critical regime.
Let us however stress that in~\eqref{def:RCMrec2}, the function $g_{\gb} = g_{\gb_n}$ also depends on $n$, but the bounds~\eqref{eq:encadrement} and the expansion~\eqref{eq:expansiong} are uniform in $n$ (as soon as~$\gb_n\to \gb_c$), so  Theorem~\ref{thm:RCMcritical} can indeed be extended to the whole near-critical regime as a consequence of Theorems~\ref{thm:convergence} and~\ref{thm:expect2}.

For simplicity of the exposition (and for later use), let us state only the case where $\lim_{n\to\infty} \gb_n =\gb_c$, but in the near-supercritical regime, that is when $m\psi_\q(\gb_n) -1 \gg \frac1n$.
Note that we have
\[
m\psi_\q(\gb_n) -1 = m(\psi_\q(\gb_n)-\psi_\q(\gb_c))\sim  \frac{(m-1)(m+\q-1)}{m \q} \, (\gb_n-\gb_c) \,, \quad \text{ as } \gb_n\to\gb_c \,,
\]
since $\psi_\q'(\gb_c) = \frac{\q\, e^{\gb_c}}{(e^{\gb_c}+\q-1)^2} = \frac{m-1}{m^2 \q} (m+\q-1)$, recalling also that $e^{\gb_c} = \frac{m+\q-1}{m-1}$, see Remark~\ref{rem:pointcritique}.
Using that when $m\psi_\q(\gb_n) -1 \gg \frac1n$ we have
\[
a_n:= \sum_{k=1}^n (m\psi_\q(\gb_n))^{-k s} \sim \frac{1}{1-(m\psi_\q(\gb_n))^{-s} } \sim  s^{-1} (m\psi_\q(\gb_n)-1)^{-1} \,,
\] 
we get
\[
a_n \sim \frac{m \q}{s (m-1)(m+\q-1)} \, (\gb_n-\gb_c)^{-1} \,, \quad \text{ as } \gb_n\to\gb_c \,.
\]
We therefore directly obtain the following result from Theorem~\ref{thm:convergence}-Proposition~\ref{prop:asymp} and Theorem~\ref{thm:expect2}.

\begin{theorem}[Near-supercritical]
\label{thm:RCMnearcritical}
Let $(\gb_n)_{n\geq 0}$ be such that $\lim_{n\to\infty} \gb_n =\gb_c$, with also $ \lim_{n\to\infty} n(\gb_n-\gb_c) =+\infty$.
Define $\pi_n(\gb_n) := \bar \PP_{\p_n,\q}^{\bar T_n}(\rho \leftrightarrow \partial T_n) $, with $\p_n = 1-e^{-\gb_n}$ (which goes to $\p_c$).
Then, letting $s=1$ if $\q\in(0,2)$ and $s=2$ if $\q=2$, we have the following:
\begin{itemize}
\item If $\bE[Z^{1+s}]<+\infty$, then we have that almost surely and in $L^{1+s}$,
\[
\lim_{n\to\infty} \frac{\pi_n(\gb_n)}{(\gb_n-\gb_c)^{1/s}} = \tilde \alpha_\q W\,,
\qquad \text{ with } \ \tilde \alpha_{\q} := \begin{cases}
\frac{2 (m+\q-1)}{\q(2-\q)}  \bE[W^2]^{-1} & \text{ if } \q\in(0,2) \,, \\[5pt]
 \sqrt{\frac{3 m}{ 2} } \,\bE[W^3]^{-1/2} &  \text{ if } \q=2
\,.
\end{cases}
\] 

\item If $c_1 L(x) x^{1+s-\alpha}\leq \bE[(Z\wedge x)^{1+s}] \leq c_2 L(x) x^{1+s-\alpha}$ for some $\alpha\in (1,1+s]$, some slowly varying function $L(\cdot)$ and constants $c_1,c_2>0$, then 
\[
\Big( \frac{\pi_n}{h^{-1}(\gb_n-\gb_c)}  \Big)_{n\geq 1} \quad \text{  is tight in $(0,+\infty)$},
\] 
with $h^{-1}(\cdot)$ an asymptotic inverse of $L(1/x) x^{\alpha-1}$ as $x\downarrow 0$.
\end{itemize}

\end{theorem}

%

Notice that in Theorem~\ref{thm:RCMnearcritical}, one can take $\gb_n\downarrow \gb_c$ arbitrarily slowly.
We can therefore deduce the following critical behavior for the limiting survival probability $\pi(\gb)$ of the random cluster model on a quenched Galton--Watson tree.

\begin{corollary}
\label{cor:RCMcritic}
Recall the definition~\eqref{def:pi} of $\pi(\gb) = \pi(\gb,T)$.
Then, letting $s=1$ if $\q\in(0,2)$ and $s=2$ if $\q=2$, we have the following critical behavior:
\begin{itemize}
\item If $\bE[Z^{1+s}]<+\infty$, then 
$\pi(\gb) \sim  \tilde \alpha_\q W \,  (\gb-\gb_c)^{1/s}$,
with $\tilde \alpha_\q$ as in Theorem~\ref{thm:RCMnearcritical}.

\item If $\bE[(Z\wedge x)^{1+s}] \asymp L(x) x^{1+s-\alpha}$, then 
$\pi(\gb) \asymp  h^{-1}(\gb-\gb_c)$, 
with $h^{-1}(\cdot)$ an asymptotic inverse of $L(1/x) x^{\alpha-1}$ as $x\downarrow 0$.
\end{itemize}

\end{corollary}

Let us stress that Corollary~\ref{cor:RCMcritic} is related to existing results.
For instance, \cite{DomGiaRvdH14} gives the critical behavior of the Ising model (with external field) on a tree-like graph. 
Our techniques are quite different, and we improve the results here by giving a sharp behavior in the case where $\bE[Z^{1+s}] <+\infty$ (with the correct \textit{random} constant); and treating a more general set-up in the heavy-tail case.
Let us also note that~\cite{MPR20} treats the survival probability of Bernoulli percolation on a Galton Watson tree and our Corollary~\ref{cor:RCMcritic} improves their main result in two directions: (i) we recover the first-order asymptotic with a weaker moment condition ($\bE[Z^2]<+\infty$ instead of $\bE[Z^{3+\eta}]<+\infty$); (ii) we treat the case of a GW with heavy tails.

\section{Some useful preliminaries on Branching processes}
\label{sec:preliminaries}

In this section, we regroup some technical tools that will be used throughout the proofs.
Some of them are standard, for instance some $L^q$ inequalities for sums of independent random variables (see Section~\ref{sec:Lqinequalities}), but we recall them for convenience.
Others are very natural, estimating the tail of (truncated) branching processes with heavy-tails (see Section~\ref{sec:truncatedBP}), but we were not able to find them in the literature so we provide a proof in Appendix~\ref{app:tail}.

\subsection{Useful $L^q$ inequalities for sums of independent random variables}
\label{sec:Lqinequalities}

Let us collect here some estimates on the $L^q$-norm of sums of independent random variables, which turn out to be extremely useful in the context of branching processes.

One of the main inequalities that we will use is Lemma~1.4 in~\cite{Liu01}, which mostly relies on the Marcinkiewicz--Zigmund inequality (see e.g.\  \cite[Ch.~10.3]{ChowTeicher}).

\begin{lemma}[Lemma~1.4 in \cite{Liu01}]
\label{lem:q-moment}
Let $q>1$ and let $(X_i)_{i\geq 1}$ be independent and \emph{centered} random variables, with a finite moment of order $q >1$. Then, for any $n\geq 1$, we have
\[
\bE\left[  \Big| \sum_{i=1}^n X_i\Big|^q \right] \leq (A_q)^q  n^{\theta_q-1}  \sum_{i=1}^n \bE\big[ |X_i|^q \big] \,,
\] 
with $\theta_q := \max(1,\frac q2)<q$ and $A_q := 2 \lceil \frac{q}{2} \rceil^{1/2}$ (in particular $A_q=2$ if $q \in (1,2]$).
\end{lemma}

Let us mention for completeness the following (simpler) result due to Neveu~\cite{Neveu87}, in the case where $q\in (1,2]$.

\begin{lemma}[\cite{Neveu87}]
\label{lem:Neveu}
Let $q\in (1,2]$. For a \emph{non-negative} r.v.\ $X$ with a finite moment of order~$q$, define $V_q(X):=\bE[X^q]-\bE[X]^q$.
Then, if $X,Y$ are independent non-negative r.v.\ with a finite moment of order $q$, we have that $V_q(X+Y) \leq V_q(X)+V_q(Y)$.

In particular, if $(X_i)_{i\geq 1}$ are independent\ \emph{non-negative} random variables with a finite moment of order $q$, then for any $n\geq 1$ we have
\[
\bE\Big[\Big( \sum_{i=1}^n X_i\Big)^q \Big] \leq \bE\Big[\sum_{i=1}^n X_i \Big]^q + \sum_{i=1}^n \bE\big[ X_i^q \big] \,.
\] 
\end{lemma}

As a direct consequence of Lemma~\ref{lem:q-moment}, we state the following lemma that will be convenient in the context of branching processes.

\begin{lemma}
\label{lem:q-moment2}
Let $q>1$, let $X$ be a \emph{non-negative} random variable and $N$ be a $\mathbb N$-valued random variable, both with a finite moment of order $q >1$.
If $(X_i)_{i\geq 1}$ are i.i.d.\ random variables with the same distribution as $X$ and independent of $N$, then we have, 
\[
 \Big\| \sum_{i=1}^N X_i \Big\|_q \leq A_q  \big(\|N\|_{\theta_q} \big)^{\frac{\theta_q}{q}} \big( \| X -\bE[X]\|_q\big) + \|N\|_q \bE[X] \,,
\] 
with $\theta_q := \max(1,\frac q2)<q$ and $A_q := 2 \lceil \frac{q}{2} \rceil^{1/2}$ as in Lemma~\ref{lem:q-moment}.
\end{lemma}

\begin{proof}
First of all, letting $\bar X_i = X_i -\bE[X]$, we have that 
\[
\Big\| \sum_{i=1}^N X_i \Big\|_q = \Big\| \sum_{i=1}^N \bar X_i  + N \bE[X]\Big\|_q \leq \Big\| \sum_{i=1}^N \bar X_i \Big\|_q  + \|N\|_q \bE[X] \,.
\]
Now, using Lemma~\ref{lem:q-moment} conditionally on $N$, we have that
\[
\bE\Big[ \Big(\sum_{i=1}^N \bar X_i\Big)^q  \Big] \leq   (A_q)^q  \bE\big[ N^{\theta_q}  \big] \bE\big[ |\bar X|^q \big] \,,
\]
so that 
$\big\| \sum_{i=1}^N \bar X_i \big\|_q \leq A_q \bE[ N^{\theta_q}]^{1/q} \|X-\bE[X]\|_q \,.
$
This concludes the proof.
\end{proof}

As a consequence of Lemma~\ref{lem:q-moment2}, we get the following bound for branching processes.
Let $(Z_k)_{k\geq 0}$ be a branching process with offspring distribution $Z\sim \mu$ which admits a moment of order $q>1$, and mean denoted by $m:=\bE[Z]$.
Denote $W_k := \frac{1}{m^k} Z_k$ the usual martingale and let $W:=\lim_{k\to\infty} W_k$ a.s., which is in $L^q$.
Then, there are two constants $c_1(q):=A_q \bE[W^{\theta_q}]^{1/q}$ and $c_2:=c_1+\|W\|_q$, such that the following holds: 
for any $k\geq 1$, if $(X_i)_{i\geq 0}$ are non-negative i.i.d.\ random variables independent of $Z_k$ (with common distribution $X$) with a finite $q$-moment, 
\begin{equation}
\label{eq:boundLqBP}
\Big\|\frac{1}{m^k} \sum_{i=1}^{Z_k} X_i \Big\|_q \leq c_1 m^{-k(1-\frac{\theta_q}{q})} \|X\|_q + c_2 \bE[X] \,.
\end{equation}
Indeed, a simple application of Lemma~\ref{lem:q-moment2} gives that
\[
\Big\|\frac{1}{m^k} \sum_{i=1}^{Z_k} X_i \Big\|_q \leq A_q m^{-k} \big(m^{k}\|W_k\|_{\theta_q} \big)^{\frac{\theta_q}{q}} \|X -\bE[X]\|_q + \|W_k\|_p\bE[X] \,.
\]
Using that $\|X -\bE[X]\|_q\leq \|X\|_q+\bE[X]$, this directly concludes~\eqref{eq:boundLqBP}, since $\sup_k\|W_k\|_{\alpha} = \|W\|_{\alpha}$ for any $\alpha\in[1,q]$ ($(W_k^{\alpha})_{k\geq 0}$ is a submartingale), and also $m^{-k(1-\frac{1}{q}\theta_q)} \leq 1$ since we have $\theta_q<q$.

\subsection{Supercritical branching processes with (truncated) heavy tails}
\label{sec:truncatedBP}

Let $\mu$ be the reproduction law of a super-critical branching process, let $Z\sim \mu$ and note $m:= \bE[Z]$ its mean.
For $t>1$ some (large) fixed parameter, we define $\tilde{Z}:= Z\wedge t$ with distribution denoted~$\tilde{\mu}$ and assume that $t$ is large enough so that $\tilde{m}:=\bE[\tilde Z] >1$. 
We define a \textit{truncated} branching process with reproduction law $\tilde{\mu} $.
As above, let $(Z_k)_{k\geq 0}$, resp.~$\tilde Z_k$, be a branching process with offspring distribution $Z\sim \mu$, resp.~$\tilde Z\sim \tilde \mu$, and denote $W_{k}:= \frac{1}{m^{k}} Z_k$ and $\tilde W_{k}:= \frac{1}{\tilde m^{k}} \tilde Z_k$ the corresponding martingales.

Notice that if $Z$ verifies~\eqref{hyp:upper}, then we have $\bE[(\tilde Z)^q] \leq c_2 L(t) t^{q-\alpha}$.
We will actually work with the following bound on the tail of $Z$ and $\tilde{Z}$ (which are equivalent to~\eqref{hyp:upper} if $\alpha < q$): for any $x \geqslant 1$
\begin{equation}
\label{eq:righttail}
\bP(Z>x) \leq c_2 L(x) x^{-\alpha},
\qquad
\bP(\tilde Z >x) \leq c_2 L(x) x^{-\alpha} \ind_{\{ x <t\}} \,. 
\end{equation}

\noindent 
We then have the following result, which is a variation of Theorem~1 (or Corollary~12) in~\cite{DKW13}.

\begin{proposition}
\label{prop:tailGW}
Assume that~\eqref{eq:righttail} holds for the offspring distribution. Then, there are constants $c,c'$ such that, uniformly in $\ell\geq 1$ 
\[
\bP(W_{\ell} > x) \leq c L(x) x^{-\alpha} \,,
\]
and, provided that~$t$ is large enough, 
\[
\bP(\tilde W_{\ell} > x) \leq 
\begin{cases}
c L(x) x^{-\alpha}  & \  \text{ for all  } x \geq 1 \,,\\
t^{- c' x/t} & \  \text{ for all } x\geq  t \,.
\end{cases}
\]
\end{proposition}

Let us mention that in~\cite{DKW13} the authors provide a uniform upper and lower bound on the tail probability $\bP(W_{\ell} > x)$, but with a slightly stronger assumption that the tail is \textit{dominated varying}, \textit{i.e.}\ $\bP(Z>x/2)/\bP(Z>x)$ remains bounded (it is for instance implied if we assume the counterpart lower bound $\bP(Z>x) \geq c_1' L(x) x^{-\alpha}$); let us stress that the proof of Lemma~11-Corollary~12 in~\cite{DKW13} actually only requires that the tail is upper bounded by a dominated varying function.
On the other hand, the case of a \textit{truncated} branching process does not seem to have been treated in the literature.
We give full (self-contained) proof of both upper bounds in Proposition~\ref{prop:tailGW} in Appendix~\ref{app:tail}.

\section{Estimates on moments of $C_n$ in the case $\bE[Z^q]<+\infty$}
\label{sec:est-finite}

\subsection{Proof of Theorem~\ref{thm:expect}}
\label{ssec:proof-thmexpect}

We decompose the proof of Theorem~\ref{thm:expect} into an upper bound, which requires only that $Z\sim \mu$ admits a finite first moment, and a lower bound, which requires a finite moment of order $q$ for $Z$. 

For $u\notin \partial T_n$, define $\phi(u) = R_u \cC_p(u \leftrightarrow \partial T_n(u))$, where we remind that $\cC_p(u \leftrightarrow \partial T_n(u))$ is the effective $p$-conductance on the subtree $T_n(u)$, equipped with the original resistances~$(R_u)_{u \in T_n}$.
We then have the following recursion: 
\begin{equation}
\label{eq:recconductance}
\phi(u) =  \sum_{v\leftarrow u} \frac{R_u}{R_v} \frac{\phi(v)}{(1+ \phi(v)^s)^{1/s}} \,.
\end{equation}

\subsubsection{Upper bound on $\bE[C_n]$}
\label{sec:upperECn}

We start from the recursion~\eqref{eq:recconductance}. We proceed as in~\cite[\S6.2.1]{AVB23} (see also~\cite[Lem.~3.1]{ChenHuLin} in the case $p=q=2$). 
Using the branching property and the fact that $R_u = R^{-|u|}$, we get that for any $u\notin \partial T_n$ and $v\leftarrow u$,
\[
\bE[\phi(u)]  = mR_n\, \bE\left[  \frac{\phi(v)}{(1+ \phi(v)^s)^{1/s}}\right] \leq mR_n \, \frac{\bE[\phi(v)]}{(1+ \bE[\phi(v)]^s)^{1/s}} \,,
\]
where we have used that the function $x \mapsto \frac{x}{(1+x^s)^{1/s}}$ is concave for the last inequality.
Denoting $w_{k} :=\bE[\phi(u)]^{-s}$ if $|u|=k$, we have the following recursive inequality:
\[
w_k \geq (mR_n)^{-s} (1+w_{k+1}).
\]
Iterating, we finally get that
\[
w_0 := \bE[C_n]^{-s} \geq \sum_{k=1}^{n} (mR_n)^{-ks} + (mR_n)^{-n s} w_n \,,
\]
Recalling that we have defined $a_n:=\sum_{k=1}^{n} (mR_n)^{-ks}$, we get that $\bE[C_n]^{-s} \geq a_n$, that is $\bE[C_n] \leq a_n^{-1/s}$.

In particular, using Markov's inequality, we have $\bP( a_n^{1/s} C_n \geq K) \leq K^{-1}$, showing the first part of the tightness of $( a_n^{1/s} C_n )_{n\geq 0}$.

\subsubsection{Lower bound on $C_n$, $\bE[C_n]$}
\label{sec:lowerEC}

For the lower bound on $C_n = \cC_p(\rho \leftrightarrow \partial T_n)$, we actually prove an upper bound for the $p$-resistance $\cR_p(\rho \leftrightarrow \partial T_n)$;
we proceed as in \cite[\S6.2.2]{AVB23}.
Using Thomson's principle~\eqref{def:thomson}, we obtain an upper bound on the $p$-resistance simply by computing the energy of a well-chosen unit flow $\theta$ from~$\rho$ to $\partial T_n$.
The uniform flow $\hat \theta$ on $T$ is a natural choice (as in~\cite[Lem.~2.2]{PemPer95}, see also \cite[Lem.~3.3]{ChenHuLin}):
\[
\hat \theta(u,v) = \frac{Z_n(v)}{Z_n} \,, 
\]
where $Z_n(v)$ is the number of descendants of $v$ at generation $n$ and $Z_n:= Z_n(\rho)$.
Then, with this unit flow  $\hat \theta (u,v)$ from the root to the leaves of $T_n$, we have
\[
\cR_p(\rho \leftrightarrow \partial T_n)^s \leq  \cE_p(\hat \theta) = \sum_{k=1}^{n} \sum_{|v|=k} (R_v)^{s} \Big( \frac{Z_n(v)}{Z_n}\Big)^q  
\]
Defining $W_n(v):= \frac{Z_n(v)}{\bE[Z_n(v)]} = m^{-(n-k)} Z_n(v) $ for $|v|=k$ and $W_n :=  W_n(\rho)$, using that $R_u=R^{-|u|}$ and recalling that $s=q-1$, we can rewrite the upper bound as
\begin{equation}
\label{eq:upperResist}
\cR_p(\rho \leftrightarrow \partial T_n)^{s} \leq  \frac{1}{(W_n)^q} \sum_{k=1}^{n} (mR_n)^{-ks}  \frac{1}{m^k} \sum_{|v|=k} W_n(v)^q \,.
\end{equation}
Now, for any $\gep>0$, we can bound 
\begin{multline*}
\bP\Big( \cR_p(\rho \leftrightarrow \partial T_n)^s   \geq \gep^{-1} a_n \Big)  \\
 \leq \bP\big( W_n \leq \gep^{1/2q} \big) + \bP\bigg(  \sum_{k=1}^{n} (mR_n)^{-ks}  \frac{1}{m^k} \sum_{|v|=k} W_n(v)^q \geq \gep^{-1/2} a_n \bigg) \,.
\end{multline*}
Now, for the first term, since the martingale $W_n$ converges a.s.\ to some non-degenerate random variable $W$ with $\bP(W >0)=1$, we get that $\bP( W_n \leq \gep^{1/2q} ) \leq \delta_{\gep}$ uniformly in~$n$, for some $\delta_{\gep} \downarrow 0$ as $\gep\downarrow 0$.
For the second probability, using Markov's inequality (together with the branching property), we get that
\[
\bP\bigg(  \sum_{k=1}^{n} (mR_n)^{-ks}  \frac{1}{m^k} \sum_{|v|=k} W_n(v)^q \geq \gep^{-1/2}  a_n \bigg) \leq \frac{\gep^{1/2}}{a_n} \sum_{k=1}^{n} (mR_n)^{-ks} \bE\big[ (W_{n-k})^q \big] \,.
\]
Assuming that $Z$ admits a finite moment of order $q$, we get that $(W_n)_{n\geq 1}$ is bounded in~$L^q$ (see~\cite{BinghamDoney}), so that there is a constant $c_q$ such that $\bE[ (W_{\ell})^q ] \leq c_q$ uniformly in $\ell\geq 0$.
This shows that the last probability is bounded by $c_q \gep^{1/2}$, recalling the definition of $a_n$.

All together, and recalling that $C_n := \cC_p(\rho \leftrightarrow \partial T_n) =  \cR_p(\rho \leftrightarrow \partial T_n)^{-1}$, we have obtained that for any $\gep >0$,
\begin{equation}
\label{eq:boundprobabCn}
\bP\big( C_n   \leq \gep^{1/s} (a_n)^{-1/s} \big) \leq \delta_{\gep} + c_q \gep^{1/2} \,,
\end{equation}
which shows the scond part of the tightness of $(a_n^{1/s} C_n)_{n\geq 0}$.

For a lower bound on $\bE[C_n]$, choose $\gep := \gep_q >0$ sufficiently small so that $\delta_{\gep} + c_q \gep^{1/2} \leq 1/2$:
we end up with
\[
\bE[C_n] \geq \frac12  \gep_q^{1/s} (a_n)^{-1/s} \,,
\]
which is the desired lower bound.

\begin{remark}
\label{rem:generalbound}
More generally, the proof shows that, for any $u \in T_n$ with $|u|=k$,
\[
c_p (a_{n-k})^{-1/s} \leq  \gp_k:=\bE[ \phi(u)] \leq  (a_{n-k})^{-1/s} \,.
\]
From this, we get that there exist constants $c_1, c_2$ such that for any $ 1\leq k\leq n$,
\[
 c_1 \leq \frac{\gp_{k}}{\gp_{k-1}} \leq c_2 \,.
\]
Indeed, we have
$a_{n-k} \leq a_{n-k+1} = (mR_n)^{-s} + (mR_n)^{-s} a_{n-k}  \leq (1+ (mR_n)^{-s}) a_{n-k}$, so that 
$c_p \leq \gp_k/\gp_{k-1} \leq c_p^{-1} (1+(mR_n)^{-s})^{1/s}$; for this, we need to assume that $\inf_{n\geq 0} R_n >0$.
\end{remark}

\subsection{Control of the higher moments of $C_n$}
\label{sec:moments}

We now give a technical result that control the moments of $C_n$, which is useful for the sequel.

\begin{proposition}
\label{prop:moments}
Let $p>1$ and let $q=\frac{p}{p-1}$ be its conjugate exponent.
If $Z\sim \mu$ admits a moment of order $r\geq q$, then so does~$C_n$. 
Additionally, if $\sup_n m R_n < m^{\max(\frac{r-1}{r},\frac12)}$ and if $\inf_n R_n>0$, then for any $r' \in (1,r]$, we have
\[
 \sup_{n\geq 1} \frac{ \bE\big[  (C_n)^{r'}\big] }{\bE[C_n]^{r'} } <+\infty \,.
\]
\end{proposition}

The proof would work similarly as in~\cite[Lem.~3.2]{ChenHuLin} if $r$ were an integer, but we aim at generalizing their result here (both to non-integer moments $r$, and to a $p$-conductance with $p\neq 2$).
We take inspiration from~\cite[Prop.~1.3]{Liu01}, which deals with the rate of convergence in~$L^q$ for the usual martingale $W_n$.

\begin{proof}
First of all, let us notice that we only have to prove the claim for $r'=r$, since then we can apply Jensen's inequality to get that $\bE[C_n^{r'}]\leq \bE[C_n^r]^{r'/r}$.
Similarly, we only need to control the $L^r$ norm of $C_n$, and show that 
\begin{equation}
\label{boundLq}
\|C_n\|_r \leq c_r \bE[C_n] \,,
\end{equation}
for some universal constant $c_r$.

Recall that we define $\phi(u)  := R_u \cC_p(u \leftrightarrow \partial T_n(u))$, with $R_u = (R_n)^{-|u|}$. 
Now, using~\eqref{eq:recconductance}, we have the general upper bound
\[
0\leq \phi(u) \leq R_n \sum_{v\leftarrow u} \phi(v)\,.
\]
First of all, notice that we easily get by iteration that $\phi(v)$, hence $C_n(v)$, admits a finite moment of order $r$.
If we iterate the above inequality for $k$ generations, we have
\[
C_n = \phi(\rho) \leq (R_n)^k  \sum_{|v|=k} \phi(v) = (R_n)^k  \sum_{i=1}^{Z_k} \phi_{k}^{(i)} \,,
\]
where $(\phi_{k}^{(i)})_{i\geq 1}$ are i.i.d.\ copies of $\phi(v)$ for $|v|=k$ (independent of $T_k$, and in particular independent of~$Z_k$).

Therefore, applying Lemma~\ref{lem:q-moment2} (more precisely~\eqref{eq:boundLqBP}) and denoting $\gp_k=\bE[\phi_k]$ (and $\gp_0 = \bE[C_n]$), we obtain that
\begin{align*}
\| C_n\|_r  = \|\phi_0\|_r  &  \leq  c_1 (m R_n)^k  m^{-k(1-\frac{\theta_r}{r})}  \|\phi_k\|_r +  c_2 (mR_n)^k \, \gp_k \\
& \leq   c_1 (m R_n)^k  m^{-k(1-\frac{\theta_r}{r})}  \|\phi_k\|_r +  c_2 2^k \, \gp_k  \,.
\end{align*}
Now, let us fix $k=k_r$ such that $\gamma_r:= (c_1)^{1/k_r} m^{ -(1-\frac{\theta_r}{r})} <1$ with $\gamma_r$ sufficiently close to $m^{ -(1-\frac{\theta_r}{r})}$ so that $\hat \gamma_{r} := \sup_{n}\gamma_r m R_n  <1$ (recall that by assumption $\sup_n mR_n < m^{1-\frac{\theta_r}{r}} $).
This way, we may write
\[
\| C_n\|_r  = \|\phi_0\|_r  \leq  ( \gamma_r m R_n)^k \| \phi_k \|_r + c'_r \gp_0  = ( \hat\gamma_r)^k \| \phi_k \|_r + c'_r \gp_0  \,,
\]
where we have also used that $\gp_{k_r} \leq (c_2)^{k_r} \gp_0$, see Remark~\ref{rem:generalbound}.
We can then iterate this inequality, applying it to $\| \phi_{k} \|_r, \| \phi_{2k}\|_r$, etc.
Letting $n= k_r n_r + j_r$ with $ 0\leq j_q \leq k_r-1$, we have
\begin{equation}
\label{eq:sumtobecontroled}
\| C_n\|_r \leq c'_r \sum_{j=0}^{n_r} (\gamma_r m R_n )^{k_r j}  \gp_{k_r j} +    c''_r (\gamma_r m R_n )^{n_r k_r} \,,
\end{equation}
with $c''_r:= c'_r \sup_{0\leq j \leq k_r-1} \|\phi_{n-j}\|_r$.

We can now estimate this last expression, using the estimate on $\gp_{k_r j} \leq c (a_{n-jk_r})^{-1/s}$ from Remark~\ref{rem:generalbound}, where we recall that $a_{\ell} := \sum_{i=1}^{\ell} (mR_n)^{-i s}$.
Hence, we need to control~$a_{\ell}$, depending on the value of $mR_n$.

\smallskip
\noindent
\textbullet\ Let us start with the case $mR_n = 1$ for simplicity. In this case, $a_\ell= \ell$ for all $\ell$, so we can bound the sum
\[
\sum_{j=0}^{n_r} (\gamma_r mR_n)^{k_r j}  \gp_{jk_r} \leq c  \sum_{j=0}^{n_r} (\hat \gamma_r)^{k_r j}  (n-jk_r)^{-1/s} \leq c' n^{-1/s} \,.
\]
Note also that  $(\hat \gamma_r)^{k_r n_r}  \leq C_r (\hat \gamma_r)^{n} \leq C_r' n^{-1/s}$, using that $\hat \gamma_r<1$.
From~\eqref{eq:sumtobecontroled}, and recalling that in this case we have $\bE[C_n] \asymp n^{-1/s}$, see~\eqref{eq:orderECn}, we end up with
\[
\| C_n\|_r  \leq  c_r \bE[C_n] \,,
\]
as desired.

\smallskip
\noindent
\textbullet\ Let us now deal with the case $mR_n > 1$. 
First of all, let us control $a_\ell$ for $1\leq \ell \leq n$.
As noticed before, $a_{\ell} = f_\ell(mR_n)$, with $f_\ell(x) = x^{-s} \frac{1-x^{-\ell s}}{1-x^{-s}}$, and we may use the following easy bounds: $f_\ell(x) \geq \ell$ if $x \in [1,1+\frac1\ell]$ and $f_\ell(x) \geq c (x-1)^{-1} $ if $x\in [1+\frac1\ell,C]$.
Hence, letting $\ell_0 := (mR_n-1)^{-1}$, we have that 
\[
\gp_{n-\ell} \leq  c
\begin{cases}
\ell^{-1/s} & \ \text{ if } \ell\leq  \ell_0 \wedge n \,,\\
(mR_n-1)^{1/s} & \ \text{ if }\ell \geq \ell_0\wedge n \,.
\end{cases}
\]
Let us consider two different cases.

First, if $mR_n \in (1,1+\frac1n]$, so that $\ell_0 \geq n$.
Then we can bound the sum
\[
\sum_{j=0}^{n_r} (\gamma_r mR_n)^{k_r j}  \gp_{k_r j} \leq c  \sum_{j=0}^{n_r } (\hat \gamma_r)^{k_r j}  (n-jk_r)^{-1/s} \leq c' n^{-1/s} \,,
\]
and we conclude as in the case $mR_n =1$ that $\| C_n\|_r  \leq  c_r \bE[C_n]$, since we also have $\bE[C_n] \asymp n^{-1/s}$ in that case, see~\eqref{eq:orderECn}.

In the case where $mR_n \in (1+\frac1n ,C]$, we have $\ell_0 \leq n$, and we need to split the sum into two parts:
\begin{align*}
\sum_{j=0}^{n_r} (\gamma_r mR_n)^{k_r j}  \gp_{k_rj} 
& \leq c (mR_n-1)^{1/s}  \sum_{j=0}^{n_r - \ell_0/k_r} (\hat \gamma_r)^{k_r j}   +  \sum_{j=n_r - \ell_0/k_r}^{n_r} (\hat \gamma_r)^{k_r j}  (n-jk_r)^{-1/s} \\
&  \leq c' (mR_n-1)^{1/s}  + c'' \sum_{i=1}^{\ell_0} (\hat \gamma_r)^{n- i}  i^{-1/s} \leq c''' (mR_n-1)^{1/s}  \,.
\end{align*}
Once more, this concludes the proof since $\bE[C_n] \asymp (mR_n-1)^{1/s}$, see~\eqref{eq:orderECn}, the last term in~\eqref{eq:sumtobecontroled} being again negligible.

\smallskip
\noindent
\textbullet\ Let us now deal with the case $mR_n < 1$. 
As above, using that $a_{\ell} = f_\ell(mR_n)$, with $f_\ell(x) = x^{-s} \frac{x^{-\ell s}-1}{x^{-s}-1}$, we may use the following easy bounds: $f_\ell (x) \geq \ell$ if $x \in [1-\frac1\ell,1)$ and $f_\ell(x) \geq c \, \frac{x^{-(\ell+1)s}}{1-x}$ if $x\in (0,1-\frac1\ell]$.
Hence, letting $\ell_0 := (1-mR_n)^{-1}$, we have that 
\[
\gp_{n- \ell} \leq  c
\begin{cases}
\ell^{-1/s} & \ \text{ if } \ell\leq \ell_0 \wedge n \,,\\
(mR_n)^{\ell+1} (1-mR_n)^{1/s}  & \ \text{ if }\ell \geq \ell_0\wedge n \,.
\end{cases}
\]
Let us again consider two different cases.
First, if $mR_n\in [1-\frac1n,1)$, so that $\ell_0 \geq n$, then we conclude that $\| C_n\|_r  \leq  c_r \bE[C_n]$ exactly as above.

In the case where $mR_n \in (0,1-\frac1n]$, we have $\ell_0 \leq n$, and we again split the sum into two parts:
\begin{multline*}
\sum_{j=0}^{n_r} (\gamma_r mR_n)^{k_r j}  \gp_{k_r j } 
 \leq c mR_n (1-mR_n)^{1/s} \sum_{j=0}^{n_r - \ell_0/k_r} (\gamma_r mR_n)^{k_r j} (mR_n)^{n- k_r j}   \\ 
+  \sum_{j=n_r - \ell_0/k_r}^{n_r} (\gamma_r mR_n)^{k_r j} (n-jk_r)^{-1/s} \,.
\end{multline*}
Now, the first sum is bounded by  a constant times
$(1-mR_n)^{1/s} (mR_n)^{n+1}  \leq c \bE[C_n]$, see~\eqref{eq:orderECn} for the second inequality.
The second term is bounded by a constant times
\[
\sum_{i=1}^{\ell_0} (\gamma_r mR_n)^{n-i} i^{-1/s} \leq c (\gamma_r mR_n)^{n-\ell_0} (\ell_0)^{-1/s} = c (mR_n)^{n-\ell_0} (1-mR_n)^{1/s} \,.
\]
Now, since $\inf_{n} mR_n >0$, we get that $(mR_n)^{-\ell_0}$ remains bounded (recall $\ell_0 = (1-mR_n)^{-1}$), so this is bounded by a constant times $\bE[C_n]$, recalling~\eqref{eq:orderECn}.

It remains to control the last term in~\eqref{eq:sumtobecontroled}. 
Again, this is bounded by a constant times $(\gamma_r)^{n} (mR_r)^{n_qk_q} \leq c (\gamma_q)^{n} (mR_n)^{n+1}$ since $\inf_n mR_n >0$; now, this is negligible compared to $\bE[C_n] \asymp (1-mR_n)^{1/s} (mR_n)^{n+1}$.
\end{proof}

\subsection{Control of the ratios $\gp_{k}/\gp_{k-1}$}
\label{sec:ratios}

In view of Remark~\ref{rem:generalbound}, we are able to bound the ratios $\gp_{k}/\gp_{k-1}$, where we recall that $\gp_k := \bE[\phi(u)]$ for $|u|=k$.
Let us now give a more precise estimate of the ratio $\gp_{k}/\gp_{k-1}$.

\begin{lemma}
\label{lem:ratios}
With the same assumption as in Proposition~\ref{prop:moments}, there exists a constant $c>0$ such that
$0\leq \frac{\gp_k}{\gp_{k-1}} - (mR_n)^{-1} \leq \frac{c}{a_{n-k}}$.
In particular, we have 
\begin{equation}
\label{eq:boundratios}
1\leq \frac{(mR_n)^k\gp_k}{\gp_0} \leq  \prod_{i=1}^k \Big(1+ \frac{c (mR_n)^i}{a_{n-i}}\Big) \,.
\end{equation}
\end{lemma}

\begin{proof}
We start with the iteration~\eqref{eq:recconductance} which defines $\phi(u)$, with $R_u = (R_n)^{-|u|}$.
Notice that we can write the iteration as
\[
\phi(u) = R_n \sum_{v\leftarrow u}  \phi(v) - R_n \sum_{v\leftarrow u}  f(\phi(v) ) \,,
\]
with $f(x) = x- \frac{x}{(1+x^s)^{1/s}}$, which verifies $0\leq f(x) \leq c \min(x^q,x)$.

Taking the expectation, we therefore get that if $|u|=k-1$,
\[
\gp_{k-1} = mR_n \gp_k - mR_n \bE[f(\phi(v))] \,.
\]
All together, using Proposition~\ref{prop:moments} to get that $\bE[f(\phi(v))] \leq C \gp_k^{q}$, we get that 
\[
  0 \leq mR_n \gp_k - \gp_{k-1} \leq c (mR_n) \gp_k^q \leq c mR_n\, \frac{\gp_k}{a_{n-k}}  \,,
\]
where we have also used that $\gp_k^{q-1} =\gp_k^s \leq (a_{n-k})^{-1}$ from Remark~\ref{rem:generalbound}.
This concludes the first bound in Lemma~\ref{lem:ratios}.
The bound~\eqref{eq:boundratios} follows immediately by iteration.
\end{proof}

\section{Convergence of the normalized conductance}
\label{sec:cvg}

We consider the recursion~\eqref{def:concaverec} on the tree of depth $n$, with resistances $R_v= (R_n)^{-|v|}$:
\[
B_n(u) = R_n \sum_{v \leftarrow u}  g\big( B_n(v)\big)\,,
\]
with $g(\cdot)$ satisfying~\eqref{eq:encadrement}. We denote $B_n:=B_n(\rho)$ and $b_n:= \bE[B_n]$.
We now show the convergence of the normalized quantity $\hat B_n := \frac{1}{b_n} B_n$ assuming that $\bE[Z^q]<+\infty$, \textit{i.e.}\ we prove Theorem~\ref{thm:convergence}.

The method we use is analogous to that in~\cite[\S5]{ChenHuLin}, with some modifications to obtain the convergence in $L^q$ with only a finite $q$-moment assumption.
We start by proving the $L^1$ convergence; then we upgrade it to a $L^q$ convergence. We then conclude the section by proving the almost sure convergence in the case $R_n\equiv R$ and the sharp asymptotic of Proposition~\ref{prop:asymp}.

\paragraph*{Preliminary observations and notation}
For $|u|=k$, let us denote $b_{n-k} := \bE[B_n(u)]$. Notice that, as $\kappa_1 C_n \leq B_n \leq \kappa_2 C_n$, we have that $b_{n-k} \asymp (a_{n-k})^{-1/s}$ with $a_{\ell} = \sum_{i=1}^{\ell} (mR_n)^{-is}$.
Notice that Proposition~\ref{prop:moments} holds, so we have that $\bE[B_n(u)^q] \leq c \bE[B_n(u)]^q$ for some universal constant $c$; we will also make use of Lemma~\ref{lem:ratios}, which is also valid for $b_{n-k}$ (the proof only uses that $0\leq f(x) :=x-g(x) \leq c \min(x^q,x)$).

\subsection{Convergence in $L^1$}

Let us write $f(x) = x - g(x)$, which is a non-negative function (by concavity of $g$), which verifies $f(x) \leq c \min(x^q,x)$ thanks to~\eqref{eq:encadrement}. Then, we can rewrite the above iteration as
\[
B_n(u) = R_n \sum_{v \leftarrow u} B_n(v) - R_n \sum_{v\leftarrow u} f(B_n(v)) \,,
\]
so that iterating for the first $k$ generations we get
\[
 B_n = (mR_n)^k \frac{1}{m^k} \sum_{|v|=k}  B_n(v) - \Pi_{k,n} \,,
\]
with
\begin{equation}
\label{def:Pikn}
\Pi_{k,n} := \sum_{j=1}^k (mR_n)^j \frac{1}{m^j}\sum_{|v|=j} f(B_n(v))  \geq 0\,.
\end{equation}
Normalizing by $b_n$, and denoting $\hat B_n(v) = \frac{1}{b_{n-k}} B_n(v)$ for $|v|=k$, we get that
\begin{equation}
\label{eq:Lpbound}
\hat B_n   = \frac{(mR_n)^k b_{n-k}}{b_n} \frac{1}{m^k} \sum_{|v|=k}  \hat B_n(v) - \frac{1}{b_n} \Pi_{k,n}  = W_k + I_{k,n}+ J_{k,n} -  \frac{1}{b_n} \Pi_{k,n}\,,
\end{equation}
where $W_k = \frac{1}{m^k} Z_k$ is the usual martingale and we have set
\[
I_{k,n}:= \frac{1}{m^k}\sum_{|v|=k} \big( \hat B_n(v)-1 \big)\quad \text{ and } \quad  J_{k,n}:=  \Big( \frac{(mR_n)^k b_{n-k}}{b_n} -1\Big)\frac{1}{m^k}\sum_{|v|=k}  \hat B_n(v) \,.
\]
We now treat the three terms in~\eqref{eq:Lpbound} separately.
To anticipate on the $L^q$ convergence, we  bound the terms $I_{k,n}$, $J_{k,n}$ in $L^q$; we then control $\frac{1}{b_n} \Pi_{k,n}$ in $L^1$.

\smallskip
\noindent
{\it Control of $I_{k,n}$.}
Since $\hat B_n(v)-1$ are i.i.d.\ centered random variables independent of $T_k$ (hence of $Z_k$), we can apply Lemma~\ref{lem:q-moment} conditionally on $Z_k$ to get that
\[
\bE\bigg[ \Big| \sum_{|v|=k} ( \hat B_n(v)-1) \Big|^q  \bigg] \leq A_q^q \bE\big[(Z_k)^{\theta_q}\big] \bE\big[ |\hat B_n(v)-1|^q \big] \,.
\]
Therefore
\[
\|I_{k,n} \|_q \leq A_q \frac{1}{m^k} \big(\|Z_k\|_{\theta_q}\big)^{\frac{\theta_q}{q}} \|\hat B_n(v) -1\|_q \leq A_q m^{-k (1-\frac{1}{q}\theta_q)} \big(\|W_k\|_{\theta_q}\big)^{\frac{\theta_q}{q}} \, \big(1+\|\hat B_n(v)\|_q\big) \,.
\]
Now, since $\theta_q\leq q$, we have that $\sup_{k} \|W_k\|_{\theta_q} <+\infty$, and thanks to Proposition~\ref{prop:moments} we get that $\|\hat B_n(v)\|_q \leq c_q$ for some universal constant $c_q$.
We therefore get that
\begin{equation}
\label{eq:controlI}
\|I_{k,n} \|_q \leq c'_q  m^{-k (1-\frac{1}{q}\theta_q)} \,,
\end{equation}
which goes to $0$ as $k\to\infty$ since $\theta_q<q$.

\smallskip
\noindent
{\it Control of $J_{k,n}$.}
First of all, we get that
 \[
 \Big\|\frac{1}{m^k} \sum_{|v|=k} \hat B_n(v)\Big\|_q \leq \| W_k \|_q + \|I_{k,n}\|_q \leq c_q \,,
 \]
for some universal constant $c_q$.
We therefore only have to focus on the term
\[
\Big| \frac{(mR_n)^k b_{n-k}}{b_n} -1\Big| \leq \prod_{i=1}^k \Big( 1+ c\, \frac{(mR_n)^i}{a_{n-i}} \Big) -1  \leq \exp\Big( c \sum_{i=1}^k \frac{(mR_n)^i}{a_{n-i}} \Big) -1 \,,
\]
where we have used Lemma~\ref{lem:ratios}. We now show that the upper bound goes to $0$ as $k\to\infty$.

\begin{claim}
\label{claim:a}
Define $k_n = \frac12 n$ if $mR_n\leq 1+ \frac1n$ and $k_n = (mR_n-1)^{-1}$ if $m R_{n} \geq 1+\frac1n$.
Then, there is a constant $c$ such that for all $k \leq k_n$, we have
\[
\sum_{i=1}^k \frac{(mR_n)^i}{a_{n-i}}  \leq c \, \frac{a_k}{a_n}
\quad \text{ and } \quad \sum_{i=1}^k \frac{1}{a_{n-i}}  \leq c \, \frac{a_k}{a_n} \,.
\]
\end{claim}

In particular, for all $k\leq k_n$ we have that $\sum_{i=1}^{k_n} \frac{(mR_n)^i}{a_{n-i}}$ is bounded by a constant $c$:  we then get that for any $k\leq k_n$
\begin{equation}
\label{eq:controlJ}
\|J_{k,n}\|_q \leq c_q \, \Big| \frac{(mR_n)^k b_{n-k}}{b_n} -1\Big| \leq  c'\sum_{i=1}^{k_n} \frac{(mR_n)^i}{a_{n-i}} \leq c'' \frac{a_k}{ a_n} \, ,
\end{equation}
which goes to $0$ if $k\to\infty$ sufficiently slowly.

\begin{proof}[Proof of Claim~\ref{claim:a}]
The case $mR_n = 1$ is trivial since then we have that $a_k=k$ and $a_n=n$.
A similar result holds when $|mR_n-1| \leq n^{-1}$, since we also have $a_k \asymp k$ for all $k\leq n$ in that case.

In the case $mR_n < 1-\frac 1n \leq 1$, we can bound $(mR_n)^i\leq 1$ in the sum so we only need to control the second sum.
Then, we can use the bound $a_{n-i} \geq c (1-mR_n)^{-1} (mR_n)^{-(n-i)s}$ for all $i\leq n/2$, see~\eqref{eq:orderan}.
We therefore get
\[
\sum_{i=1}^k \frac{1}{a_{n-i}} \leq c (1-mR_n) (mR_n)^{-ns} \sum_{i=1}^k (mR_n)^{-is} \leq c \frac{a_k}{a_n} \,.
\]

In the case $mR_n > 1+\frac 1n \geq 1$, we can bound $(mR_n)^i\geq 1$ so we only need to control the first sum. 
Using that $a_{n-i} \geq c (mR_n -1)^{-1}  \geq c a_n$ for all $i\leq n/2$, see~\eqref{eq:orderan}, we get that 
\[
\sum_{i=1}^k \frac{(mR_n)^i}{a_{n-i}} \leq \frac{c}{a_n} \sum_{i=1}^k (mR_n)^{i} \leq \frac{c k}{a_n} \,,
\]
provided that $k\leq (mR_n-1)^{-1}$, since $(mR_n)^i \leq C$ uniformly for $i\leq (mR_n-1)^{-1}$.
Note that we have $a_k \asymp k$ for $(mR_n-1)^{-1}$, which concludes the proof.
\end{proof}

\smallskip
\noindent
{\it Control of $\frac{1}{b_n} \Pi_{k,n}$.}
Recalling the definition~\eqref{def:Pikn} of $\Pi_{k,n}\geq 0$ we get that
\[
\frac{1}{b_n} \bE[\Pi_{k,n}] = \frac{1}{b_n} \sum_{j=1}^k (mR_n)^j \bE[f(B_n(v))] \leq  c  \sum_{j=1}^k  \frac{(mR_n)^j b_{n-j}}{b_n} \, (b_{n-j})^s  \,.
\]
For the last inequality, we have used that $f(x) \leq c\, \min(x^q,x)$, so that for $|v|=j$ we have $\bE[f(B_n(v))] \leq c \bE[B_n(v)^q] \leq c' (b_{n-j})^q$, thanks to Proposition~\ref{prop:moments}; recall also that $s=q-1$.

Again, we can use Lemma~\ref{lem:ratios} and Claim~\ref{claim:a} to get that for all $j\leq k_n$
\[
 \frac{(mR_n)^j b_{n-j}}{b_n} \leq \prod_{i=1}^j \Big( 1+ c\, \frac{(mR_n)^i}{a_{n-i}} \Big) \leq 1 + c'' \frac{a_k}{ a_n} \leq C \,.
\]
Using also that $(b_{n-j})^{s} \leq c (a_{n-j})^{-1}$ by Remark~\ref{rem:generalbound}, we therefore get that for 
\begin{equation}
\label{eq:controlPi}
\frac{1}{b_n}\bE[\Pi_{k,n}] \leq c' \sum_{j=1}^k  \frac{1}{a_{n-j}} \leq c'' \, \frac{a_k}{a_{n}} \,,
\end{equation}
using again Claim~\ref{claim:a} for the last inequality.

\smallskip
\noindent
\textit{Conclusion.}
Going back to~\eqref{eq:Lpbound} and collecting the bounds~\eqref{eq:controlI}-\eqref{eq:controlJ}-\eqref{eq:controlPi}, we obtain that for all $k\leq k_n$ (with $k_n$ defined in Claim~\ref{claim:a}), we have
\begin{equation}
\label{eq:L1bound2}
\begin{split}
\bE\big[ |\hat B_n - W | \big]
& \leq \bE\big[ |W_k- W | \big] +  c m^{-k(1-\frac{1}{q}\theta_q)} +  c\, \frac{a_k}{a_n} \\
& \leq c' m^{-k(1-\frac{1}{q}\theta_q)} +  c\, \frac{a_k}{a_n}
\,,
\end{split}
\end{equation}
For the last inequality, we have used that $\bE[|W-W_k|] \leq \|W-W_k\|_q \leq c m^{-k(1-\frac{1}{q}\theta_q)}$, see \cite[Prop.~1.3]{Liu01}.

Note that the assumption that $\limsup_{n\to\infty} mR_n \leq 1$ ensures that $k_n$ goes to $+\infty$ and that $\lim_{n\to\infty} a_n = +\infty$, see~\eqref{eq:orderan}.
Hence, we can choose $k=\hat k_n \leq k_n$ going to $+\infty$ sufficiently slowly so that the upper bound in~\eqref{eq:L1bound2} goes to zero.
This concludes the proof that $(\hat B_n)_{n\geq 0}$ converges in~$L^1$ to $W$.
\qed

\subsection{Convergence in $L^q$}

Since we have the convergence $\hat B_n \to W$ in $L^1$, we also have the convergence in probability.
To prove the convergence in $L^q$, we therefore simply need to show the uniform integrability of $(\hat B_n^q)_{n\geq 0}$.

But from~\eqref{eq:Lpbound}, we have the upper bound $\hat B_n \leq W_{k} + I_{k,n} + J_{k,n}$, where we can choose $k=\hat k_n \leq k_n$ going to infinity slowly enough.
Since $\hat B_n \geq 0$, we therefore get that
\[
0\leq \hat B_n^q \leq 3^q (W_{\hat k_n})^q + 3^q(I_{\hat k_n,n})^q + 3^q (J_{\hat k_n,n})^q \,,
\]
and we only have to prove the uniform integrability of the three terms on the right-hand side, which is easy.

First,~\eqref{eq:controlI}  and \eqref{eq:controlJ} show that $(I_{\hat k_n,n})^q$ and $(J_{\hat k_n,n})^q$ converge to $0$ in $L^1$; in particular they are uniformly integrable.
Second, since $\mu$ admits a finite moment of order $q$, we have that $(W_k)_{k\geq 0}$ converges in $L^q$ to $W$, so in particular $(W_k^q)_{k\geq 0}$ is uniformly integrable.
This concludes the proof that $(\hat B_n^q)_{n\geq 0}$ is uniformly integrable, hence that $(\hat B_n)_{n\geq 0}$ converges in~$L^q$ to $W$.
\qed

\subsection{Almost sure convergence}

First of all, notice that if we take $k = \hat k_n = c \log n$ in~\eqref{eq:L1bound2}, with a constant $c$ sufficienly large,
we obtain that
\begin{equation}
\label{eq:L1bound-ps}
\bE\left[ \big|\hat B_n - W \big| \right] \leq c' n^{-2} + c \frac{a_{\hat k_n} }{a_n} \,.
\end{equation}
Hence, if $k_n$ from Claim~\ref{claim:a} satisfies $k_n \geq c \log n = \hat k_n$ and if  $a_{\hat k_n} /a_n$ is summable, we directly obtain the a.s.\ convergence $\lim_{n\to\infty} \hat B_n = W$.

This is for instance the case if we have $R_n \equiv R$ with $R \in (0,m^{-1}]$ since in that case $k_n =\frac12 n$ and $a_n \asymp (m R)^{-n s}$; more generally, it is verified if $mR_n \leq 1- c \frac{\log n}{n}$ with some constant $c$ large enough.
This settles the a.s.\ convergence when $R_n\equiv R \in (0,m^{-1})$, and it remains to treat the critical case $R_n\equiv m^{-1}$.


Let $R_n \equiv m^{-1}$, so that $a_k =k$ for all $k\geq 1$, and in particular  $\frac{a_{\hat k_n} }{a_n} \leq c \frac{\log n}{n}$.
Consider the subsequence $(n^2)_{n\geq 0}$, so that the upper bound in~\eqref{eq:L1bound-ps} is summable along this subsequence.
Then~\eqref{eq:L1bound-ps} gives the a.s.\ convergence $\lim_{n\to\infty} \hat B_{n^2} = W$. 
Hence, one simply needs to bridge the gaps between $n^2$ and $(n+1)^2$.
Now, notice that  $(B_{n})_{n\geq 1}$ is non-increasing  (this is where we use that $R_n\equiv R$ is fixed, see Remark~\ref{rem:Rnconstant} below), so we can write that, for all~${n^2 \leq \ell <  (n+1)^2}$
\[
 \frac{b_{(n+1)^2}}{b_{n^2}} \hat B_{(n+1)^2} \leq \hat B_{\ell} \leq \frac{b_{n^2}}{b_{(n+1)^2}} \hat B_{n^2} \,,
\]
and it only remains to show that $b_{(n+1)^2}/b_{n^2}$ goes to $1$. 
But this simply comes from Lemma~\ref{lem:ratios}, which shows that
\[
1 \leq \frac{b_{n^2+k}}{b_{n^2}} \leq \prod_{i=1}^{k} \Big( 1 + \frac{C}{a_{n^2+i}} \Big) \leq \exp\Big( C \sum_{i=1}^k \frac{1}{a_{n^2+i}} \Big) \leq \exp\Big( C \frac{k}{n^2} \Big) \,,
\]
using also that $a_{n^2+i} \geq  n^2$ for all $i\geq 1$.
Therefore, since $(n+1)^2 = n^2 + 2n+1$, we get that 
$1 \leq \frac{b_{(n+1)^2}}{b_{n^2}} \leq \exp\big( C \frac{2n+1}{n^2}\big)$ and therefore goes to $1$ as $n\to\infty$,
concluding the proof that $\lim_{n\to\infty} \hat B_n =W$ almost surely.
\qed

\begin{remark}
\label{rem:Rnconstant}
One could try to use the idea of taking a subsequence to adapt the proof to a general sequence $(R_n)_{n\geq 0}$, 
but the difficulty is to compare $B_n$ and $B_{n+1}$, since $B_n$ uses resistances $(R_n)^{-|v|}$ inside $T_n$ and $B_{n+1}$ resistances $(R_{n+1})^{-|v|}$ inside $T_{n+1}$; the restriction to the case $R_n\equiv R$ allows for a comparison.
Similarly, in the general case, there is no obvious relation between $b_{n}:=\bE[B_n]$ and $b_{n+1}:=\bE[B_{n+1}]$.
\end{remark}

\subsection{Precise asymptotic for $\bE[B_n]$: proof of Proposition~\ref{prop:asymp}}

Theorem~\ref{thm:convergence} proves that~$\hat B_n$ converges in $L^q$ to $W$, so we get that
\[
\lim_{n\to\infty} \bE[(\hat B_n)^q] = \bE[W^q] \,.
\]
This shows in particular that $\bE[B_n^q] \sim \bE[W^q] (b_n)^q$.

Note that, setting $f(x)= x-g(x)$, we have the relation
$b_{n} = mR_n \big(  b_{n-1} - \bE[ f( B_n(v))   ]\big)$ for $|v|=1$,
which can be rewritten as
\[
b_{n} = mR_n b_{n-1} (1 - c_{n-1}) \,,\quad  \text{ with } \ c_{n-1} := (b_{n-1})^{-1}\bE[ f( B_n(v)) ] \,.
\]
Now, observe that since $f(x) \sim \kappa_g x^q$ as $x\downarrow 0$ and that $f(x) \leq c x^q$, we get by dominated convergence that 
$\bE[ f( B_n(v)) ] \sim \kappa_g \bE[ B_n(v)^q] \sim \kappa_g \bE[W^q] (b_{n-1})^q$, with $b_{n-1} \downarrow 0$ under the assumption of Proposition~\ref{prop:asymp}. 

In particular, this gives that $c_{n-1}\sim \kappa_g \bE[W^q] (b_{n-1})^s$, with $s=q-1$.
Setting $x_n = b_{n}^{-s}$, we get that
\[
x_n =  (mR_n)^{-s} x_{n-1} (1  + c'_{n-1}) 
= (mR_n)^{-s} x_{n-1}  + (mR_n)^{-s} x_{n-1} c'_{n-1}
\,,
\]
with $c'_{n-1} = (1-c_{n-1})^{-s} -1  \sim s \kappa_g \bE[W^q] (b_{n-1})^s$.
Iterating this relation, we get that
\[
x_n = \sum_{k=1}^{n} (mR_n)^{-ks} x_{n-k} c'_{n-k} \,,
\]
with $x_{i}c'_i \to s \kappa_g \bE[W^q]$ as $i\to\infty$.
Hence, we can choose $\ell_n$  going to $+\infty$ arbitrarily slowly and write that
\[
\begin{split}
x_n & = (1+o(1)) s \kappa_g \bE[W^q] \sum_{k=1}^{n -\ell_n} (mR_n)^{-ks}  + (mR_n)^{-ns}\sum_{i=0}^{\ell_n} (mR_n)^{is} x_{i}c'_i \\
 & = (1+o(1)) s \kappa_g \bE[W^q]  a_{n-\ell_n} + (mR_n)^{-ns} \sum_{i=0}^{\ell_n} (mR_n)^{is} x_{i}c'_i \,.
 \end{split}
\]

To conclude, let us notice that, in the case where $\lim_{n\to\infty} mR_n = \vartheta=1$, then we can choose $\ell_n\to\infty$ so that $a_{n-\ell_n} \sim a_n$. We also find in that case that the second term is negligible compared to $a_n$ provided that $\ell_n$ grows sufficiently slowly  ---~this is clear if $mR_n \geq 1-\frac{c}{n}$ and follows from the fact that $a_n \sim (1-(mR_n)^s)^{-1} (mR_n)^{-(n+1)s}$ if $n(mR_n-1) \to -\infty$.
This shows that $x_n \sim s \kappa_g \bE[W^q] a_n$ as $n\to\infty$ when $\vartheta=1$.

On the other hand, if $\lim_{n\to\infty} mR_n = \vartheta \in (0,1)$, then $a_n \sim (1-\vartheta^{s})^{-1} (mR_n)^{-ns}$, so for any $\ell_n \to\infty$ we have $a_{n-\ell_n} =o(a_n)$.
We end up with $x_n \sim  c_{\vartheta} (mR_n)^{-ns}$, with $c_\vartheta = \sum_{i=0}^{+\infty} \vartheta^{is} x_{i}c'_i$ which is a convergent sequence.
This concludes the fact that $x_n \sim c_{\vartheta}(1-\vartheta^{s}) a_n$ as $n\to\infty$.

All together, this gives the desired conclusion, since $x_n \sim (b_n)^{-s}$.
\qed

\section{Estimates on moments of $C_n$ in the case $\bE[Z^q]=+\infty$}
\label{sec:est-infinite}

In this section, we prove Theorem~\ref{thm:expect2}. 
Since we can bound $B_n$ with $C_n := \cC_p(\rho \leftrightarrow \partial T_n)$, see Proposition~\ref{prop:boundCB}, we focus on estimates on $C_n$, $\bE[C_n]$.
Let us define, as in Section~\ref{ssec:proof-thmexpect}, $\phi(u) = R_u \cC_p(u \leftrightarrow \partial T_n(u))$ and recall that we have the following recursion~\eqref{eq:recconductance}:
\begin{equation}
\label{eq:recconductance2}
\phi(u) = R_n \sum_{v\leftarrow u} g \big(\phi(v)\big) \qquad  \text{ with } g(x) = \frac{x}{(1+x^s)^{1/s}} \,.
\end{equation}

\subsection{Upper bound on $\bE[C_n]$}

Let us assume that~\eqref{hyp:lower} holds and obtain the upper bound in Theorem~\ref{thm:expect2}.
As in Section~\ref{sec:moments}, let $\gp_k=\bE[\phi(u)]$ for $|u|=k$, so in particular $\gp_n=1$ and we need to estimate $\gp_0$.
Using the recursion~\eqref{eq:recconductance2} we can write, again with the notation~${f(x)=x-g(x)}$,
\begin{equation}
\label{eq:expgphi}
\begin{split}
\bE[ g(\phi(u))] & = mR_n \bE[g(\phi(v))]- \bE\Big[ f\Big( R_n \sum_{v\leftarrow u} g \big(\phi(v)\big) \Big) \Big] \\
& =  mR_n \bE[g(\phi(v))]- \bE\Big[f\Big( R_n Z_u \bE\big[ g \big(\phi(v)\big)\big] \Big) \Big]  \\
& \qquad \qquad \qquad \qquad \qquad  - \bE\bigg[f\Big( R_n \sum_{v\leftarrow u} g \big(\phi(v)\big) \Big)  - f\Big( R_n Z_u \bE\big[ g \big(\phi(v)\big)\big] \Big)\bigg] \,,
\end{split}
\end{equation}
where we have denoted $Z_u$ the number of descendants of $u$. 
Notice that, for $v\in T_n$ such that $|v|=k+1$, we have $\gp_k = mR_n \bE[ g(\phi(v))]$ and that, by convexity of $f$, the last term in~\eqref{eq:expgphi} is non-positive: we end up with the following inequality: for all $1\leq k\leq n$
\[
\frac{1}{mR_n}  \gp_{k-1} \leq   \gp_{k}  - \bE\Big[f\Big( \frac1m Z  \gp_{k} \Big) \Big] \,.
\]
Now, we can use the fact that $f(x) \geq c \min(x^q,x) \geq c(x\wedge 1)^q$, to get that
\[
\bE\Big[f\Big( \frac1m Z_u \gp_k \Big) \Big] \geq c' (\gp_{k})^q \bE\Big[ \Big(Z \wedge \frac{m}{\gp_{k}} \Big)^q \Big] \geq c' L(1/\gp_{k}) (\gp_{k})^{-\alpha} \,.
\]
where we have used assumption~\eqref{hyp:lower} for the last inequality.

All together, we end up with the following recursion: $\gp_n=1$ and for $0\leq k\leq n$
\begin{equation} 
\gp_{k-1} \leq mR_n \gp_{k} (1-h( \gp_{k}))  \,,
\end{equation}
where $h(x) \sim c L(1/x) x^{\alpha-1}$ as $x\downarrow 0$.
Note that we can assume that both $x\mapsto h(x) \in [0,1)$ and $x\mapsto x(1-h(x))$ are increasing (by properties of regularly varying functions, we may assume that $h'(x)= c' L(1/x) x^{\alpha-2}$ so $1-h(x)-xh'(x)$ remains positive).

We can therefore focus on the iteration 
\begin{equation}
\label{eq:recurtildephi}
u_{k+1} = mR_n u_{k} (1-h(u_{k})) \qquad \text{ with } h(x) \sim c L(1/x) x^{\alpha-1} \,,
\end{equation}
started at $u_0 = 1$. We now have to obtain an upper bound on $u_{n}$.

\begin{remark}
Notice that the recursion~\eqref{eq:recurtildephi} admits a non-zero fixed point $u_*$ if one has $mR_n>1$, and that it verifies $h(u_*) = \frac{mR_n-1}{mR_n}$. If on the other hand $mR_n\leq 1$, the only fixed point is $u_*=0$. 
\end{remark}

\smallskip
\noindent
\textbullet\ 
Let us start with the case where $mR_n \in [1+\frac1n,K]$, and let $\delta_n := \frac{mR_n-1}{mR_n}$.
We let $\upsilon_n$ be the fixed point of the equation~\eqref{eq:recurtildephi}, \textit{i.e.}\ such that $h(\upsilon_n) = \delta_n$, and notice that it verifies $\upsilon_n= h^{-1}(\delta_n) \leq c\gamma_n$, where $\gamma_n$ is defined in~\eqref{def:gamma}.
Let us assume that $\upsilon_n \leq u_0$, otherwise we have $u_k \leq \upsilon_n$ for all $k$ and in particular $u_n\leq \upsilon_n \leq c\gamma_n$.

Now, by assumption we have that $(u_k)_{k\geq 0}$ is a decreasing sequence. 
Let us define $k_n:= \min\{ k : h(u_k) \leq  C \delta_n\}$ for some (large) constant $C\geq 2$.
Our goal is to show that if $C$ is large enough then we have $k_n\leq n$, so in particular $u_n\leq u_{k_n}$ with $u_{k_n}\leq h^{-1}(C\delta_n) \leq  c' \gamma_n$.

Now, for all $k < k_n$, we have that $mR_n(1-h(u_{k})) \leq 1+\delta_n - h(u_k) \leq 1- \frac12 h(u_k)$, so we end up with
\[
u_{k+1} \leq  u_k \big(1- \tfrac12 h(u_k) \big) \,.
\]
Now, let $H: (0,\infty) \to (0,\infty)$ be some decreasing function, with derivative given by $H'(x) = - (x h(x))^{-1}$; we also let $c>0$ be a constant such that $H(x - t) \geq H(x) - c t H'(x)$ for all $x\in (0,1]$ and $t \in [0,x/2]$.
We then have that for $k<k_n$,
\[
H(u_{k+1}) \geq H\big( u_k -\tfrac12 u_k h(u_k) \big)
\geq H(u_k) -  c u_k H'(u_k) h(u_k) = H(u_k) + c \,.
\]
All together, we get that for all $k\leq k_n$
\begin{equation*}
H(u_{k})-H(u_0) = \sum_{j=0}^{k-1} \big( H(u_{j+1}) -H(u_{j}) \big) \geq  c k \,, 
\end{equation*}
or, put otherwise $H(u_{k}) \geq c k$.
Since $H(u) = \int_u^{1} (t h(t))^{-1} \dd t $ and since $h(t)$ is regularly varying with index $\alpha-1>0$, we get that $H(u) \sim  \frac{1}{\alpha-1} h(u)^{-1}$ as $u\downarrow 0$, so we end up with the fact that, for all $k\leq k_n$
\begin{equation}
\label{eq:huk}
h(u_{k})^{-1}  \geq   c' \, k \quad \text{ or } \quad h(u_{k})  \leq   c'' \, k^{-1} \,,\quad \text{ for all } k\leq k_n \,. 
\end{equation}
Now, applying this inequality with $k=k_n-1$ and recalling the definition of $k_n$, we get that $k_n-1 \leq c''C^{-1} / \delta_n$.
Since $\delta_n = (mR_n-1)/mR_n \geq 1/Kn$ for $mR_n\in[1+\frac{1}{n},K]$, we get that $k_n-1 \leq K c'' C^{-1} n$, which is smaller than $n$ provided that $C$ had been fixed large enough.

\smallskip
\noindent
\textbullet\ The proof is analogous in the case where $mR_n \in [1-\frac1n,1+\frac1n]$.
Define $k_n:= \min\{ k, h(u_k) \leq 2/n \}$, so that as above we have $mR_n(1-h(u_{k})) \leq 1- \frac12 h(u_k)$ for all $k < k_n$.

Then, similarly as in~\eqref{eq:huk}, we get that $h(u_k) \leq c''k^{-1}$ for all $k\leq k_n$.
Now, either we have $k_n\geq n$, in which case $h(u_n)\leq c''/n$, or we have $k_n<n$ in which case by definition of $k_n$ we have $h(u_n) \leq h(u_{k_n}) \leq 2/n$.
In any case we have that $u_{n} \leq h^{-1}(c/n)$, so that we obtain $u_n \leq c \gamma_n$, where $\gamma_n$ is defined in~\eqref{def:gamma}, recalling also that $a_n\asymp 1/n$.

\smallskip
\noindent
\textbullet\ Let us now treat the case where $mR_n \leq 1-\frac1n$, with $\inf_{n}R_n>0$.
Let us set $\delta_n = 1-mR_n \geq \frac{1}{n}$ and let us define define $k_n:= \min\{ k, h(u_k) \leq  C \delta_n\}$ for some (large) constant~$C$.

First, let us show that $k_n<n$, provided that $C$ has been fixed large enough.
For $k < k_n$ we use the inequality $u_{k+1} \leq u_k(1-h(u_k))$, which thanks to~\eqref{eq:huk} gives that $h(u_{k})\leq c''/k$ for all $k\leq k_n$ and in particular $k_n-1\leq c'C^{-1}/\delta_n$. 
Since $\delta_n \geq 1/n$, this proves that $k_n <n$ provided that $C$ is large.

Then, for $k\geq k_n$, we use the bound $u_{k+1}\leq mR_n u_k$ to get that
\[
u_{n} \leq (mR_n)^{n-k_n} u_{k_n} \leq (mR_n)^n (1-\delta_n)^{-k_n} h^{-1}(C\delta_n)  \,,
\]
where we have also used the definition of $k_n$ for the last inequality.
Now, since we have $k_n \leq c'/\delta_n$, the term $(1-\delta_n)^{-k_n}$ remains bounded, while we have $h^{-1}(C\delta_n)\leq c h^{-1}(\delta_n)$.
This concludes the proof of the upper bound in Theorem~\ref{thm:expect2}.

The last inequality simply comes from Markov's inequality, which gives $\bP( C_n \geq K \gamma_n) \leq K^{-1} \gamma_n \bE[C_n] \leq c K^{-1}$.

\subsection{Lower bound on $C_n$, $\bE[C_n]$}

To obtain a lower bound on $C_n$, $\bE[C_n]$, let us assume that~\eqref{hyp:upper} holds.
We let $(t_n)_{n\geq 1}$ be a truncation sequence (to be optimized later on), and as in \ref{sec:truncatedBP}, we consider a \textit{truncated} branching process $\tilde T$ with offspring distribution $\tilde Z := (Z\wedge t_n) \sim \tilde{\mu}$.
Let also $\tilde C_n(v)$ be the $p$-conductances associated with the Galton--Watson tree $\tilde T_{n}$ of depth $n$ with reproduction law $\tilde \mu$.
Since we have truncated the offspring distribution there is a coupling for which $\tilde T \subset T$ and since the function $g(x) = \frac{x}{(1+x^s)^{1/s}}$ is increasing, we obtain that $\tilde C_n (v) \leq C_n(v)$ for all $v\in \tilde T$.
In particular, we only need to obtain a lower bound on $\tilde C_n$, $\bE[\tilde C_n]$.

We use the same method as in Section~\ref{sec:lowerEC}. By using the uniform flow and Thompson's principle, we have similarly to~\eqref{eq:upperResist}
\[
\tilde \cR_p(\rho \leftrightarrow \partial \tilde T_n)^s \leq \frac{1}{(\tilde W_n)^q} \sum_{k=1}^n (\tilde m R_n)^{-ks} \frac{1}{\tilde m^k} \sum_{|v|=k} \tilde W_n(v)^q \,,
\]
where $\tilde W_n(v) = \tilde m^{-(n-k)} \tilde Z_{n}(v)$ for $|v|=k$, $\tilde W_n =\tilde W_n(\rho)$.
Then, exactly as in~\eqref{eq:boundprobabCn}, we obtain the following bound for $C_n$:
\[
\bP\big( C_n \leq \gep^{1/s} (\tilde a_n)^{-1/s}\big) \leq \delta_{\gep} + \gep^{1/2} \,,
\]
where 
\begin{equation}
\label{def:atilde}
 \tilde a_n  :=  \sum_{k=1}^n (\tilde m R_n)^{-ks}  \bE\big[(\tilde W_{n-k})^q\big] \,.
\end{equation}
It only remains to show that $(\tilde a_n)^{-1/s} \geq c \tilde{\gamma}_n$, with $\tilde{\gamma}_n$ defined in~\eqref{def:gamma}.
We now need to obtain an upper bound on~\eqref{def:atilde}.
The lower bound on $\bE[\tilde C_n] \geq c' \tilde{\gamma}_n$ then follows immediately.

\smallskip
\noindent
{\it Step 1. Estimate of $\bE[(\tilde W_{\ell})^p]$.}
Let us prove that, under the assumption~\eqref{hyp:upper}, there is a constant $c$ (independent of $t_n$) such that for all $\ell\geq 1$,
\begin{equation}
\label{eq:boundtildeW}
\bE\big[ (\tilde W_{\ell})^q \big] \leq c  L(t_n) t_n^{q-\alpha} \,.
\end{equation}

Let us start with the case where $\alpha \in (1,q)$, which can be treated easily with Proposition~\ref{prop:tailGW}. 
Indeed, recalling Remark~\ref{rem:tail-moment}, when $\alpha<q$, we have the bounds~\eqref{eq:righttail} on the tail of $Z$.
Then provided that $t_n$ is large enough,  Proposition~\ref{prop:tailGW} shows that $\bP(\tilde W_{\ell} > x) \leq c L(x) x^{-\alpha}$ for all $x\geq 1$ and $\bP(\tilde W_{\ell} > x) \leq t_n^{-\alpha} e^{-c'x/2t_n}$ for $x\geq 2\alpha t_n/c'$ (assume also that $t_n\geq e$).
Then, we obtain
\[
\begin{split}
\bE\big[(\tilde W_{\ell})^q \big] & = q \int_0^{\infty} x^{q-1} \bP(\tilde W_{\ell} > x) \dd x  \\
& \leq c \int_0^{2 \alpha t_n/c'} L(x) x^{q-\alpha-1} \dd x  +  t_n^{-\alpha} \int_{2\alpha t_n/c'}^{\infty} x^{q-1} t_n^{ \alpha -c' x/t_n}\dd x \,.
\end{split}
\]
Then, by a simple change of variable $u=\frac{x}{t_n}$, we get 
\[
 \bE\big[(\tilde W_{\ell})^q \big] \leq  c'' L(t_n) t_n^{q-\alpha}  + c  t_n^{1-\alpha} \int_{\alpha/c'}^{\infty} u^{q-1} e^{- c'u/2} \dd u  \leq c'''  L(t_n) t_n^{q-\alpha} \,,
\]
using that the last integral is finite and $1-\alpha < q-\alpha$.

It remains to show~\eqref{eq:boundtildeW} in the case where~\eqref{hyp:upper} holds with $\alpha=q$.
We actually provide a proof that works as long as $\alpha \in (\frac{q}{2}, q]$.
%
For any $\ell\geq 0$ and $k\geq 0$, let us write
$\tilde W_{\ell+k} = \frac{1}{\tilde m^k}  \sum_{i=1}^{\tilde Z_k} \tilde W_\ell^{(i)}$,
where $(\tilde W_\ell^{(i)})$ are i.i.d.\ copies of $\tilde W_{\ell}$, independent of $\tilde Z_k$.
Then, since $\bE[\tilde W_\ell] =1$ for all $\ell$, we can write
\[
\tilde W_{\ell+k} -1 = \frac{1}{\tilde m^k}  \sum_{i=1}^{\tilde Z_k} \big( \tilde W_\ell^{(i)} -1) + \big(\tilde W_k-1 \big)\,.
\]
Then, by Lemma~\ref{lem:q-moment},
we have that
\[ 
\bE\bigg[\Big| \sum_{i=1}^{\tilde Z_k} \big( \tilde W_\ell^{(i)} -1)\Big|^q \; \Big|\; \tilde Z_k \bigg]\leq \frac{A_q^q}{\tilde m^{kq}}  (\tilde Z_k)^{\theta_q} \bE\Big[ \big|\tilde W_{\ell} -1\big|^q \Big] \,,
\]
so that
\begin{equation}
\label{eq:iternormp}
\big\|\tilde W_{\ell+k} -1 \big\|_q \leq \frac{A_q}{\tilde m^k} \bE\big[ (\tilde Z_k)^{\theta_q} \big]^{1/q} \big\| \tilde W_{\ell} -1\big\|_q  + \|W_k-1\|_q \,.
\end{equation}
Note that if $\bE[Z^{\theta_q}]<+\infty$, in particular if $\alpha > \theta_q = \max(1,\frac{q}{2})$, then we have 
$\bE[ (\tilde Z_k)^{\theta_q}] \leq \bE[(Z_k)^{\theta_q}] = m^{k \theta_q}\bE[ ( W_k)^{\theta_q}]$
with $\bE[ ( W_k)^{\theta_q}]$ bounded by a constant.
All together, we obtain that 
\[
\big\|\tilde W_{\ell+k} -1 \big\|_q \leq c_q \Big(\frac{m^{\theta_q/q}}{\tilde m}\Big)^k \big\| \tilde W_{\ell} -1\big\|_q  + \|\tilde W_k-1\|_q \,.
\]
Now, since $\theta_q<q$, we can choose $C$ large enough so that $m^{\theta_q/q} < \tilde m$ uniformly for $t_n \geq C$, and then we can fix $k_q$ large enough  so that $c_q (m^{\theta_q/q}/\tilde m)^{k_q} < 1$.
Then,  iterating the above inequality, we obtain  that there is a constant such that $\big\|\tilde W_{ik_q} -1 \big\|_q \leq C' \|\tilde W_{k_q}-1\|_p$ for all $i\geq 0$.
Since~$k_q$ is fixed, we can also apply~\eqref{eq:iternormp} recursively to get that $\|\tilde W_{k_q}-1\|_q \leq C'' \|\tilde W_1-1\|_q$.
Since $\bE[(\tilde W_j)^p]$ is non-decreasing, we therefore end up with
\[
\sup_\ell \|\tilde W_{\ell}\|_q = \sup_i \|\tilde W_{i k_q}\|_q \leq C'  +  C'' \big(\|\tilde W_1-1\|_q\big) \,, 
\]
so in particular $\bE[(\tilde W_{\ell})^q] \leq C + C' \bE[\tilde Z^q]$ for all $\ell \geq 1$. 
Using~\eqref{hyp:upper}, this concludes the proof of~\eqref{eq:boundtildeW}.

\smallskip
\noindent
{\it Step 2. Estimate of $\tilde a_n$: proof that $(\tilde a_n)^{-1/s} \geq c \tilde \gamma_n$.}
With~\eqref{eq:boundtildeW} at hand and recalling the definition~\eqref{def:atilde} of $\tilde a_n$, we have the following upper bound
\begin{equation}
\label{eq:tildegamma}
\tilde a_n \leq  c L(t_n) t_n^{q-\alpha}\sum_{k=1}^n (\tilde m R_n)^{-k s} 
= c t_n^s h(1/t_n)\sum_{k=1}^n (\tilde m R_n)^{-k s}  \,,
\end{equation}
with $h(x) \sim L(1/x) x^{\alpha-1}$ as $x\downarrow 0$ (recall $s=q-1$).

First of all, notice that we have the following identity for $\tilde m$:
\[
 \tilde m  = \tilde m(t_n)= \bE[ Z \wedge t_n ] = m- \bE[(Z-t_n)\ind_{\{Z>t_n\}}]  \geq m - c L(t_n) t_n^{1-\alpha} \,,
\]
where for the last inequality we have used that $\bP(Z>x) \leq c_2 L(x) x^{-\alpha}$ (this is always valid, see Remark~\ref{rem:tail-moment}), so that
$\bE[(Z-t_n)\ind_{\{Z>t_n\}}] = \int_{t_n}^{\infty} \bP(Z>x) \dd x \leq c L(t_n) t_n^{1-\alpha}$.

\smallskip
Let us now introduce the quantity $\tilde \delta_n = \tilde \delta_n(t_n)$ defined by $\tilde m(t_n) = (1-\tilde \delta_n)\, m$ and notice that 
\[
\tilde \delta_n \leq c L(t_n) t_n^{1-\alpha} = c h(1/t_n) .
\]

\smallskip
\noindent
\textbullet\ Let us start with the case $mR_n \in [1-\frac{1}{n},1+\frac{1}{n}]$.
We then choose $t_n$ such that $h(1/t_n) = 1/n$ so that $\tilde \delta_n \leq c/n$.
We then get that $\tilde m R_n \geq  (1-\frac1n) (1- \delta_n)\geq 1-\frac{c'}{n}$, and there exists a constant $C$ such that $(\tilde m R_n)^{-k s} \leq C$ uniformly for $k\leq n$.
All together, we get that
\[
\tilde a_n \leq  c\, t_n^{s} h(1/t_n)  \sum_{k=1}^n C = c C\, t_n^{s} \,.
\]
We conclude that $(\tilde a_n)^{-1/s}\geq c/t_n$ with $1/t_n = h^{-1}(1/n) = \gamma_n =\tilde \gamma
_n$, recalling~\eqref{def:gamma}.

\smallskip
\noindent
\textbullet\ We now treat the case $mR_n \in [1+\frac1n,K]$.
Choose $t_n$ such that $h(1/t_n) = c (mR_n-1)$ with $c$ small enough so that $\tilde \delta_n \leq \frac14(mR_n-1) $.
Then, we get that $\tilde m R_n = mR_n (1-\tilde \delta_n) \geq 1  + \frac12 (mR_n-1)$, using that $(1+x)(1-\frac14x) \geq 1+\frac12x$ for $x\in [0,1]$. 
We end up with 
\[
\tilde a_n \leq  c\, t_n^{s} h(1/t_n)  \sum_{k=1}^n \Big(1+\frac{mR_n-1}{2}\Big)^{-ks} \leq  c t_n^{s} (mR_n-1) \frac{1}{1- (1  +  \frac{mR_n-1}{2})^{-s}} \leq c' t_n^s \,.
\]
We conclude that $(\tilde a_n)^{-1/s}\geq c/t_n$ with $1/t_n = h^{-1}(c (mR_n-1)) \geq c \gamma_n = c\tilde\gamma_n$, recall~\eqref{def:gamma}.

\smallskip
\noindent
\textbullet\ Finally, we treat the case $mR_n \in (0,1-\frac1n]$.
Simply using that $\tilde m R_n <1$, we have that
\[
\tilde a_n \leq  c\, t_n^{s} h(1/t_n)   \frac{(\tilde m R_n)^{-(n+1)s}-1}{(\tilde m R_n)^{-s}-1}  \leq c' t_n^s \frac{h(1/t_n)}{1-m R_n} (m R_n)^{-n s} (1-\tilde \delta_n)^{-ns} 
\]
where we have used that $(\tilde m R_n)^{-s} -1 \geq (m R_n)^{-s} -1 \geq c (1-mR_n)$ and also that $\tilde m =  (1-\tilde \delta_n) m$.
Since we have $\tilde \delta_n \geq c h(1/t_n)$, we get that $(1- \tilde \delta_n)^{-ns} \leq \exp(c n h(1/t_n))$, which pushes us to choose $t_n$ such that $h(1/t_n)=1/n$.
In particular, with this choice, $(1-\tilde \delta_n)^{-ns} \leq c$ and $\frac{h(1/t_n)}{1-m R_n} = (n(1-mR_n))^{-1}$.

Then, we have
\[
(\tilde a_n)^{-1/s}\geq c t_n^{-1} (m R_n)^{n} \,\big( n(1-mR_n) \big)^{1/s} 
\] 
so that $(\tilde a_n)^{-1/s} \geq \tilde \gamma_n$, recalling that $t_n^{-1} = h^{-1}(1/n)$ and the definition of~$\tilde \gamma_n$.

This concludes the proof.
\qed

\begin{appendix}

\section{About the iteration of the random cluster model on trees}
\label{sec:RCM-proofs}

In this section, we prove Proposition~\ref{prop:RCM}.
For $u\in T_n$, let us introduce the event 
\[
A_u :=\{ \exists \text{ open path from } u \text{ to } \partial T_n(u) \text{ inside } T_n(u)\} \,.
\]
Denoting $E_n(u)$ the set of edges in $T_n(u)$, we also introduce the following notation: for any $A \subset \{0,1\}^{E_n(u)}$, let
\begin{equation}
\label{def:ZA}
Z_{u}(A) = Z_{\p,\q}^{T_n(u)}(A) := \sum_{\go \in A} \p^{o(\go)} (1-\p)^{f(\go)} \q^{k(\go)} = \EE_\p\big[ \q^{K_u} \ind_A \big]  
\end{equation}
be the partition function of the $(\p,\q)$-RCM model on $T_n(u)$ restricted to the event $A$; we also denote $Z_u$ for the partition function with $A=\{0,1\}^{E_n(u)}$.
Here, we have rewritten the partition function using $\PP_\p$ the distribution of the usual percolation model with parameter $\p$, \textit{i.e.}\ under $\PP_\p$ the random variables $(\go_e)_{e\in E_n}$ are i.i.d.\ $\mathrm{Bern}(\p)$, and $K_u$ is the random variable that counts the number of percolation clusters in $\bar T_n(u)$.

We are now going to show the following relations between $Z_u(A_u), Z_{u}(A_u^c)$ and $Z_{v}(A_v), Z_{v}(A_v^c)$ for $v\leftarrow u$: denoting $d_u := |\{ v, v\leftarrow u\}|$ the number of children of~$u$, we have
\begin{align}
\label{eq:recZ1}
Z_u(A_u^c) & = \q^{2-d_u}\prod_{v\leftarrow u} \Big( (1-\p) Z_v  +  \p\q^{-1} Z_v(A_v^c) \Big)  \,,\\ 
\q Z_u(A_u)+Z_u(A_u^c) & = \q^{2-d_u} \prod_{v\leftarrow u} \Big( Z_v  + \p(\q^{-1}-1) Z_v(A_v^c) \Big) \,.
\label{eq:recZ2}
\end{align}
The key observation is to write a relation between the number of clusters in $\bar T_n(u)$ and those in $(\bar T_n(v))_{v\leftarrow u}$: for every $\go\in \{0,1\}^{E_n}$, we have
\begin{equation}
\label{eq:relclusters}
K_u = \sum_{v\leftarrow u} K_v - (d_u-1) + 1  - \sum_{v\leftarrow u} \ind_{\{\go_{uv}=1\}} \ind_{\{A_v^c\}} - \ind_{A_u}
\end{equation}
Indeed, counting first the clusters in $\bar T_n(u)$ if all edges $uv$, $v\leftarrow u$ are open, we get that the number of clusters in $\bar T_n(u)$ is the sum of the number of clusters in $(\bar T_n(v))_{v\leftrightarrow}$ minus $d_u-1$, because the wired boundary condition contracts the $d_u$ clusters attached to each $\partial \bar T_n(v)$ to a single cluster, plus $1$ to include the cluster of the root $u$.
Then, adding the edges $uv$ for which~${\go_{uv}=1}$, this reduces the number of cluster by
\begin{itemize}
\item $\ind_{A_v^c}$ each time that $\go_{uv}=1$, since having $\go_{uv}$ connects two clusters that are not already connected through $\partial T_n(u)$ (which is wired);

\item one (only once) on the event $A_u = \bigcup_{v\leftrightarrow u} \{\go_{uv}=1\}\cap A_v$, since then the root $u$ is connected to $\partial T_n(u)$ so adding more than one open edge $\go_{uv}=1$ with $A_v$ will not decrease further the number of clusters.
\end{itemize}

\noindent
We refer to Figure~\ref{fig:recursion} for an illustration of the relation~\eqref{eq:relclusters}.

\begin{figure}
\begin{center}
\includegraphics[scale=1]{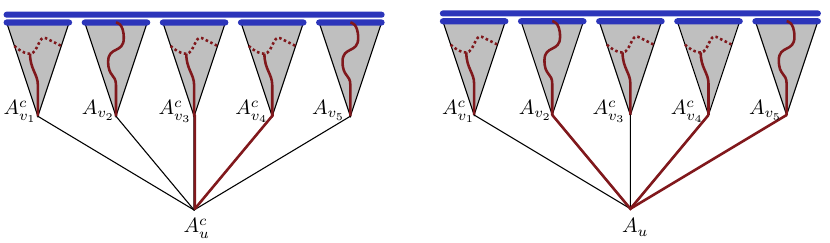}
\end{center}
\caption{\footnotesize Illustration of the identity~\eqref{eq:relclusters} that relates the numbers of clusters in $(\bar T_n(v))_{v\leftarrow u}$ to the number of clusters in $\bar T_n(u)$. 
We have illustrated two cases. 
On the left, the event $A_u$ is not verified \textit{i.e.}\ $u\not \leftrightarrow \partial T_n(u)$: there are open edges $uv$ (in thick red) only connecting to subtrees $\bar T_n(v)$ where $A_v$ is not verified, and each of these edges reduces the number of clusters by one.
On the right, the event $A_u$ is verified \textit{i.e.}\ $u \leftrightarrow \partial T_n(u)$: there are open edges $uv$ (in thick red) connecting to subtrees $\bar T_n(v)$ where $A_v$ is verified and all these open edges reduce the global number of clusters only by one (other open edges connecting to subtrees $\bar T_n(v)$ where $A_v$ is not verified still reduce the number of clusters each by one).
}
\label{fig:recursion}
\end{figure}

Let us now prove~\eqref{eq:recZ1}-\eqref{eq:recZ2}.
Starting from~\eqref{def:ZA} and since the event $ A_u^c$ can be written as $\bigcap_{v\leftarrow u} \big( \{\go_{uv}=0\} \cup ( \{\go_{uv}=1\} \cap A_v^c) \big)$, we have that
\[
\begin{split}
Z_u(A_u^c) &= \q^{2-d_u} \EE_p\Big[ \ind_{A_u^c} \prod_{v\leftarrow u} \q^{K_v} \q^{- \ind_{\{\go_{uv}=1\}} \ind_{A_v^c}}  \Big] \\
& =\q^{2-d_u} \EE_p\Big[ \prod_{v\leftarrow u} \q^{K_v} \big( \ind_{\{\go_{uv} =0\}} + \q^{-1} \ind_{\{\go_{uv}=1\}} \ind_{A_v^c}  \big)  \Big] \,,
\end{split}
\]
where we have also used the relation~\eqref{eq:relclusters} for the first identity.
Using the independence of the $\go_{uv}$ under $\PP_\p$, we get~\eqref{eq:recZ1}.

On the other hand, starting again from~\eqref{def:ZA} and the relation~\eqref{eq:relclusters}, we get that
\[
\q Z_u(A_u) = \q^{2-d_u} \EE_p\Big[ \ind_{A_u} \prod_{v\leftarrow u} \q^{K_v} \q^{- \ind_{\{\go_{uv}=1\}} \ind_{A_v^c}}  \Big]  \,.
\]
Writing $\ind_{A_u} = 1-\ind_{A_u^c}$, and recognizing the formula for $Z_u(A_u^c)$ from above, we get that 
\[
\begin{split}
\q Z_u(A_u) + Z_u(A_u^c)
& =\q^{2-d_u} \EE_p\Big[ \prod_{v\leftarrow u} \q^{K_v} \q^{- \ind_{\{\go_{uv}=1\}} \ind_{A_v^c}}  \Big] \\
& = \q^{2-d_u} \EE_p\Big[ \prod_{v\leftarrow u} \q^{K_v} \big( 1 + (\q^{-1}-1) \ind_{\{\go_{uv}=1\}} \ind_{A_v^c}  \big)  \Big] \,.
\end{split}
\]
Again, using the independence of the $\go_{uv}$ under $\PP_\p$, we get~\eqref{eq:recZ2}.

Now, from~\eqref{eq:recZ1}-\eqref{eq:recZ2}, noticing that $Z_u(A_u) =\pi_n(u) Z_u$ and $Z_u(A_u^c) = (1-\pi_n(u)) Z_u$, we get that
\begin{align*}
1- \pi_n(u) & = \q^{2-d_u} \frac{\prod_{v\leftarrow u} Z_v}{Z_u}  \times \prod_{v\leftarrow u} \Big( 1-\p+\p\q^{-1}  -  \p\q^{-1} \pi_n(v) \Big) \\
1+(\q-1)\pi_n(u)  & = \q^{2-d_u} \frac{\prod_{v\leftarrow u} Z_v}{Z_u} \times \prod_{v\leftarrow u} \Big( 1-\p+\p\q^{-1} + \p\q^{-1} (\q-1) \pi_n(v) \Big) \,,
\end{align*}
so that dividing the first line by the second one we get that
\[
\frac{1-\pi_n(u)}{1+(\q-1)\pi_n(u)} = \prod_{v\leftarrow u} \frac{1 - \gamma_{\p,\q} \pi_n(v)}{1+ \gamma_{\p,\q} (\q-1)  \pi_n(v)} \,,
\]
with $\gamma_{\p,\q} := \frac{\p\q^{-1}}{1-\p+\p\q^{-1}}  = \frac{\p}{\p+\q(1-\p)}$.
This concludes the proof of~\eqref{rec:pinu}.

Setting $b_n(u):=-\log \gp_\q(\pi_n(u))$, inverting the relation (and noticing that $\gp_\q^{-1} =\gp_\q$) we easily get that $\pi_n(u) = \gp_\q( \exp(-b_n(u)))  = \psi_\q(b_n(u)) $
\[
b_n(u) = \sum_{v\leftarrow u} -\log \gp_\q \big( \psi_\q (\gb) \psi_\q( b_n(v)) \big) \,,
\]
where we have also written $\gamma = \gp_{\q}(1-\p) = \psi_\q(\gb)$.
To conclude the proof of~\eqref{def:RCMrec}, it only remains to observe that $\psi_\q^{-1}(x) = - \log \gp_\q(x)$, since $\psi_\q(t) = \gp_\q(e^{-t})$ and $\gp_{\q}^{-1} = \gp_\q$.

\section{Branching processes with (truncated) heavy tails}
\label{app:tail}

This section is devoted to the proof of Proposition~\ref{prop:tailGW}.
Recall that we assume that 
\[
\bP(Z>x) \leq L(x) x^{-\alpha} \,, \qquad  \bP(\tilde Z>x) \leq L(x) x^{-\alpha} \ind_{\{x<t\}} \,,
\]
where $t$ is a fixed truncation parameter that we assume to be large.

\smallskip
We start with the proof of the first inequality in Proposition~\ref{prop:tailGW}: there is a constant $c>0$
\begin{equation}
\label{eq:tail1}
\bP(W_{\ell} >x) \leq c L(x) x^{-\alpha} \qquad \text{ for all } x\geq 1 \,.
\end{equation}
(We recall that this inequality could be deduced from the proof of~\cite{DKW13}, but we provide here a self-contained proof for completeness.)
We start with a first general lemma --- then, the proof relies on an inductive use of this lemma.

\begin{lemma}
\label{lem:tail}
Let $(X_i)_{i\geq 0}$ be independent non-negative random variables, with common mean $\bE[X_i]=1$ and which all satisfy
\begin{equation}
\label{eq:condtail}
\bP(X_i>x) \leq \kappa_1 L(x) x^{-\alpha} \,, \qquad \text{ for } x\geq 1 \,.
\end{equation}
Let $N$ be a $\NN$-valued random variable with mean $\mu>1$, independent of the $X_i$'s.
Then there is some constant $c>0$ and $\delta>0$ (depending only on $\kappa_1$, $\alpha$ and $L(\cdot)$) such that, for $x\geq 1$,
\[
\bP\Big(\frac{1}{\mu} \sum_{i=1}^{N} X_i >x  \Big) \leq  \bP\big(  N >  (1-\mu^{-\delta}) \mu x \big) + c  \mu^{-\frac12 \alpha} \bP\big( N> \tfrac12 \mu x\big) +  c \mu^{- \frac12 (\alpha-1)}  L(x) x^{-\alpha}  \,.
\]
\end{lemma}

\begin{proof}
Let $x\geq 1$ and let us set $\ell_1 := \frac12 \mu x$ and $\ell_2= (1 - \mu^{-\delta}) \mu x$ with $\delta>0$ fixed such that $\delta < \frac12 \wedge \frac{\alpha-1}{2\alpha}$.
Then, we can split the probability as
\[
\begin{split}
\bP \Big(\frac{1}{\mu} \sum_{i=1}^{N} X_i >x  \Big) & = \bigg( \sum_{\ell=1}^{ \ell_1} +  \sum_{\ell= \ell_1 +1}^{\ell_2}+ \sum_{\ell> \ell_2} \bigg) \bP(N=\ell) \bP\Big( \sum_{i=1}^{\ell} X_i > \mu x \Big) \\
&  =: \ T_1\ +\ T_2\ +\ T_3 \,.
\end{split}
\]

Notice already that for the last term $T_3$, by definition of $\ell_2$, we simply have
\[
T_3 \leq \bP\big( N >  \ell_{2} \big)  = \bP\big(  N > (1-\mu^{-\delta}) \mu x \big) \,.
\]
We now treat the remaining two terms. 
We use the following so-called one-big-jump behavior for sums of heavy tailed random variables, see e.g.\ \cite{Nagaev79,DDS08} (or \cite[Thm.~5.1 \& Eq. (5.1)]{Ber19b} for a more convenient formulation, close to~\eqref{eq:onebigjump} below).
We have the following statement: assuming~\eqref{eq:condtail}, there is a constant $c>0$ such that, uniformly for $\ell$ such that $\mu x-\ell \geq \ell^{1-\delta}$, we have
\begin{equation}
\label{eq:onebigjump}
\bP\Big( \sum_{i=1}^{\ell} X_i > \mu x \Big) = \bP\Big( \sum_{i=1}^{\ell} (X_i -\bE[X_i]) > \mu x-\ell \Big)  \leq c \ell\, \kappa_1 L(\mu x-\ell) (\mu x-\ell)^{-\alpha} \,.
\end{equation}

\smallskip
\noindent
{\it Term $T_1$.}
Using~\eqref{eq:onebigjump}, recalling that $\ell_1=\frac12 \mu x$ so that in particular $\mu x- \ell \geq \frac12 \mu x$, the first term is bounded by
\[
T_1 \leq c  L(\mu x) (\mu x)^{-\alpha} \sum_{\ell=1}^{ \frac12 \mu x }  \ell \bP(N=\ell) \leq c' \mu^{-\alpha+\delta}  L(x) x^{-\alpha}  \bE[N] \,,
\]
using Potter's bound \cite[Thm. 1.5.6]{BGT89}.
Since $\bE[N]=\mu$ and since we chose $\delta < \frac{1}{2} (\alpha-1)$, we get that
\[
T_1\leq c  \mu^{-\frac12 (\alpha-1)} L(x) x^{-\alpha} \,.
\]

\smallskip
\noindent
{\it Term $T_2$.}
Using~\eqref{eq:onebigjump}, recalling that $\ell_1=\frac12 \mu x$ and $\ell_2 = (1-\mu^{-\delta}) \mu x$ (which verifies $\mu x -\ell_2 = \mu^{1-\delta} x\geq (\mu x)^{1-\delta} \geq \ell_2^{1-\delta}$), we get that
\[
T_2\leq 
\bP\big( N >  \ell_1  \big)
\bP\Big( \sum_{i=1}^{\ell_2} X_i > \mu x \Big) \leq \bP\big(  N >  \tfrac12 \mu x  \big) \times c \kappa_1 L((\mu x)^{1-\delta}) (\mu x)^{-(1-\delta)\alpha} \,.
\]
All together, using again Potter's bound \cite[Thm. 1.5.6]{BGT89} and since $\delta <1/2$ and $x\geq 1$ we end up with
\[
T_2 \leq c \mu^{-\alpha/2}\bP\big( N > \tfrac12  \mu x  \big) \,.
\]
This concludes the proof.
\end{proof}

Let us now obtain~\eqref{eq:tail1} simply by iterating Lemma~\ref{lem:tail}; note that our assumption is that~$W_1$ satisfies the tail condition~\eqref{eq:condtail}.
We write that $W_{\ell+1} = \frac{1}{m^{\ell}} \sum_{i=1}^{Z_{\ell}} W_1^{(i)}$, where $(W_1^{(i)})_{i\geq 1}$ are i.i.d.\ non-negative random variables with mean one, independent of $Z_{\ell}$.
Let $c_\ell:=\prod_{i=1}^{\ell} (1-m^{-\delta\ell})^{-1}$ so that we have $c_{\ell+1}(1-m^{-\delta\ell}) =c_{\ell}$.
Then, applying Lemma~\ref{lem:tail} and noting that $\mu =\bE[Z_{\ell}]=m^{\ell}$, we get that for any $x\geq 1$
\begin{equation}
\label{eq:iterationWell}
\begin{split}
\bP\big( W_{\ell+1} \geq c_{\ell+1} x \big) & \leq \bP\big( W_{\ell} \geq c_{\ell} x \big) + c  m^{-\frac12 (\alpha-1)\ell} L(x) x^{-\alpha} +c m^{- \alpha\ell/2} \bP\big(W_{\ell} \geq  \tfrac12 c_{\ell} x\big) \,,  
\end{split}
\end{equation}
where we also used that $c_{\ell}$ is bounded by a universal constant $c_{\infty} <+\infty$ to get the bound $L(c_{\ell+1} x) (c_{\ell+1}x)^{-\alpha} \leq c L(x) x^{-\alpha}$.
Iterating~\eqref{eq:iterationWell}, we get 
\[
\bP\big( W_{\ell} \geq c_{\ell} x \big)  
\leq \kappa_{\ell}  L(x)x^{-\alpha} \,,
\]
with $\kappa_{\ell+1} = \kappa_{\ell} + c  m^{-\frac12 (\alpha-1)\ell} + c' \kappa_{\ell} m^{-\alpha\ell/2}$.
All together, since $\kappa:=\sup_{\ell \geq 1} \kappa_{\ell} <+\infty$, and $c_{\infty} := \sup_{\ell\geq 1}c_{\ell} <+\infty$, we get~\eqref{eq:tail1} (up to a change in the constants).

\smallskip
We now turn to the second inequality of Proposition~\ref{prop:tailGW}, on the martingale $\tilde W_{\ell}$ associated with the truncated branching process.
First, notice that we also have $\bP(\tilde Z>x) \leq c L(x) x^{-\alpha}$ for all $x\geq 1$, so by the first part of the Proposition we have that 
$\bP(\tilde W_{\ell} >x) \leq c L(x) x^{-\alpha}$, uniformly in $\ell\geq 1$.
It therefore remains to prove the last inequality, \textit{i.e.} that there is a constant $c'>0$ such that, provided that $t$ is large enough,
\begin{equation}
\label{eq:tail2}
\bP(\tilde W_{\ell} > x) \leq t^{- c' x/t} \qquad \text{ for all } x\geq t\,,
\end{equation}
uniformly in $\ell\geq 1$.
For this, we use the following standard Chernov's bound: for any $\lambda>0$, we have
\[
\bP(\tilde W_{\ell} > x) \leq e^{-\lambda x} \bE\big[ e^{\lambda \tilde W_{\ell}} \big] \,.
\]
Then, our next task is to bound the Laplace transform of $\tilde W_{\ell}$, uniformly in $\ell$.
We prove that there exists some $c >0$ such that, for all $\ell\geq 1$, provided that $t$ is large enough,
\begin{equation}
\label{eq:Laplace}
\bE\big[ e^{ c \frac{\log t}{t} \, \tilde W_{\ell}} \big] \leq 2 \,.
\end{equation}
Plugged into the above and choosing $\lambda = \frac{c}{t} \log t$, we end up with $\bP(\tilde W_{\ell} \geq x) \leq 2 t^{-c x/t}$ for all $x\geq 1$, which proves~\eqref{eq:tail2}.
It remains to prove~\eqref{eq:Laplace}, and we rely on the following Lemma.
\begin{lemma}
\label{lem:tailtrunc}
Let $(X_i)_{i\geq 0}$ be independent non-negative random variables, with common mean $\bE[X_i]=1$. Let $y>1$ and assume that $\bP(X_i\in [0,t])=1$ and let $\sigma^2(t)$ be such that $\bE[X_i^2]\leq \sigma^2(t)$ for all~$i$.
Let $N$ be a $\NN$-valued random variable with mean $\mu>0$, independent of the $X_i$'s.
Then, for all $\lambda >0$,
\[
\bE\bigg[ \exp\bigg(  \frac{\lambda}{\mu} \sum_{i=1}^N X_i \bigg) \bigg] \leq  \bE\bigg[ \exp\bigg( \Big(1+ \frac{\lambda}{\mu} \sigma^2(t) e^{t\lambda/\mu} \Big) \times  \frac{\lambda}{\mu}N \bigg) \bigg]\,.
\]
\end{lemma}

\begin{proof}
First of all, notice that taking first the conditional expectation with respect to~$N$, we obtain
\[
\bE\Big[ \exp\Big( \lambda \frac{1}{\mu} \sum_{i=1}^N X_i \Big) \Big]
= \bE\bigg[ \prod_{i=1}^N \bE\Big[ e^{\frac{\lambda}{\mu} X_i} \Big]  \bigg] \,.
\]
We now control the term 
\begin{equation}
\label{eq:expandLaplace}
\bE\Big[ e^{\frac{\lambda}{\mu} X_i} \Big] \leq 1 + \frac{\lambda}{\mu} + \frac{\lambda^2}{\mu^2} \bE[X_i^2] e^{\frac{\lambda}{\mu} t} \leq \exp\Big( \frac{\lambda}{\mu} + \frac{\lambda^2}{\mu^2} \bE[X_1^2] e^{\lambda t/\mu} \Big) \,,
\end{equation}
where we used that $e^{x} \leq 1+x+ x^2 e^{t x}$ for all $x\in[0, t]$ and the fact that $\bE[X_i] =1$.
Bounding $\bE[X_i^2]\leq \sigma^2(t)$ and combined with the previous identity, this gives the desired result.
\end{proof}

We are now ready to conclude the proof of~\eqref{eq:Laplace}, using Lemma~\ref{lem:tailtrunc} iteratively with the recurrence relation
$\tilde W_{\ell+1} = \frac{1}{\tilde m^\ell} \sum_{i=1}^{\tilde Z_{\ell}} \tilde W_{1}^{(i)}$.
Recalling that $\tilde W_1 \in[0, t]$ almost surely, Lemma~\ref{lem:tailtrunc} then gives that, for any $\lambda>0$,
\begin{equation}
\label{eq:forLaplace}
\bE\Big[ e^{\lambda \tilde W_{\ell+1}} \Big] \leq \bE\Big[ e^{ (1+ \tilde m^{-\ell}\gep_\lambda)\, \lambda \tilde W_{\ell}} \Big] \,,
\end{equation}
with $\gep_\lambda := \lambda \bE[(\tilde W_1)^2] e^{\lambda t}$ (bounding also $e^{\lambda t/\tilde m^\ell} \leq e^{\lambda t}  $).

Using the fact that we have $\bP(\tilde Z >x) \leq c_2 L(x)x^{-\alpha} \ind_{\{x < t\}}$, we obtain that: if $\alpha\in (1,2)$, then $\bE[(\tilde W_1)^2]\leq c_2' L(t) t^{2-\alpha}$; if $\alpha \geq 2$, then $\bE[\tilde Z]\leq \hat L(t_n)$ for $\alpha \geq 2$, for some slowly varying function (a constant if $\alpha>2$). 
We can therefore write that 
\[
\gep_{\lambda} \leq c \lambda \hat L(t) t_n^{2- \alpha\wedge2} e^{\lambda t} \,,
\]
for some slowly varying function $\hat L(\cdot)$.
Now, if we take $\lambda \leq c_{\alpha} \frac{\log t}{t}$ with a constant $c_{\alpha}:=\frac12 (\alpha\wedge2 -1)>0$, we get that $\gep_{\lambda} \leq c_{\alpha} \log t \hat L(t) t^{-2c_{\alpha}} t^{c_{\alpha}}$, so in particular $\tilde \gep_{\lambda} \leq 1$, provided that~$t$ is large enough.

Let us fix $\lambda = \delta \frac{\log t}{t}$ with $\delta := c_{\alpha} \prod_{i\geq 1}(1+\tilde m^{-i})^{-1}$, and let us define $(\lambda_k)_{1\leq k \leq \ell}$ by setting $\lambda_1=\lambda$ and $\lambda_{k+1}:= (1+ \tilde m^{k-\ell}) \lambda_{k}$ for $1\leq k\leq \ell-1$. Notice that $\lambda_k \leq \lambda_{\ell} \leq c_{\alpha} \frac{\log t}{t}$; in particular $\gep_{\lambda_{k}} \leq 1$ for all $k\in \{1,\ldots, \ell\}$.
Applying~\eqref{eq:forLaplace} iteratively, we then get that 
\[
\bE\Big[ e^{\lambda  \tilde W_{\ell}} \Big]=
\bE\Big[ e^{\lambda_1  \tilde W_{\ell}} \Big] \leq \bE\Big[ e^{\lambda_2 \tilde W_{\ell-1}} \Big] \leq \cdots \leq \bE\Big[ e^{\lambda_{\ell} \tilde W_{1}} \Big] \,.
\]
Similarly to~\eqref{eq:expandLaplace}, this last expression is bounded by
\[
1+ \lambda_\ell + \lambda_\ell^2 \bE[(\tilde W_1)^2] e^{\lambda_\ell t} \leq 1+ \lambda_\ell (1+\gep_{\lambda_\ell}) \leq 1+2\lambda_\ell  \leq 2\,,  
\]
using again that $\gep_{\lambda_{\ell}}\leq 1$ and then that $\lambda_\ell \leq c_{\alpha} \frac{\log t}{t} \leq 1$ if $t$ is large enough.
This concludes the proof of~\eqref{eq:Laplace} and thus of the second part of Proposition~\ref{prop:tailGW}.
\qed

\end{appendix}

\paragraph*{Acknowledgement} We would like to thank Remco van der Hofstad, Piet Lammers, Arnaud Le Ny, Yueyun Hu and Raphaël Rossignol for several interesting and enlightening discussions on the subject; we also thank Christina Goldschmidt and Markus Heydenreich for communicating to us that they obtained the (same) constants in Theorem~\ref{thm:RCMcritical} in an unpublished work.
Finally, we are grateful for the insightful comments of two anonymous referees, who helped us improve the presentation of the paper.
Q. Berger acknowledges the support of grant ANR-22-CE40-0012.

\bibliographystyle{abbrv} 
\bibliography{biblio.bib}

\end{document}